\documentclass[a4paper]{amsart}

\usepackage{amsmath,amssymb,amsbsy}
\usepackage{amsthm}
\usepackage{tikz}
\usepackage[all]{xy}
 \usepackage{verbatim}
\usepackage[toc,page]{appendix}
\usepackage{hyperref}
\usepackage{color, framed}
\hypersetup{urlcolor=blue, citecolor=red}

\newtheorem{theorem}{Theorem}[section]
\newtheorem{proposition}[theorem]{Proposition}
\newtheorem{corollary}[theorem]{Corollary}
\newtheorem{lemma}[theorem]{Lemma}
\newtheorem{main lemma}[theorem]{Main Lemma}

\theoremstyle{definition}

\newtheorem{remark}{Remark}

\newcommand{\R}{\mathbb{R}}

\def\eps{\varepsilon}
\def\pa{\partial}

\DeclareMathOperator{\loc}{loc}
\DeclareMathOperator{\interior}{int}

\DeclareMathOperator{\pv}{P.V.}

\title[Overdetermined fractional BVP in exterior sets]{Overdetermined problems for the fractional Laplacian in exterior and annular sets}
\author{Nicola Soave and Enrico Valdinoci}

\address{
\hbox{\parbox{5.7in}{\medskip\noindent
 Nicola Soave\\
Mathematisches Institut, Justus-Liebig-Universit\"at Giessen, \\
Arndtstrasse 2, 35392 Giessen, Germany,\\[2pt]
{\em{E-mail address: }}{\tt nicola.soave@gmail.com, nicola.soave@math.uni-giessen.de.}\\[5pt]
Enrico Valdinoci\\
Weiestrass Institut f\"ur Angewandte Analysis und Stochastik, \\
Mohrenstrasse 39, 10117 Berlin, Germany, \\[2pt] 
and \\[2pt]
Dipartimento di Matematica, Universit\`a degli studi di Milano, \\
Via Saldini 50, 20133 Milan, Italy, \\[2pt]
and \\[2pt]
Istituto di Matematica Applicata e Tecnologie Informatiche,
Consiglio Nazionale delle Ricerche, \\
Via Ferrata 1, 27100 Pavia, Italy. \\[2pt]
{\em{E-mail address: }}{\tt enrico.valdinoci@wias-berlin.de}}}}

\thanks{
The first author is partially supported through the project ERC Advanced Grant 2013 n. 339958 ``Complex Patterns for Strongly Interacting Dynamical Systems - COMPAT". The second author is partially supported through 
the project ERC Starting Grant 2011 n. 277749 ``Elliptic PDEs and Symmetry of Interfaces and Layers for Odd Nonlinearities - EPSILON" and
and the project PRIN Grant 201274FYK7
``Aspetti variazionali e
perturbativi nei problemi differenziali nonlineari''. }

\begin{document}

\begin{abstract}
We consider a fractional elliptic equation
in an unbounded set with both Dirichlet and fractional normal derivative datum
prescribed. We prove that the domain and the solution are necessarily
radially symmetric. 

The extension of the result in bounded non-convex regions is also studied, as well as the radial symmetry of the solution when the set is a priori supposed to be rotationally symmetric.
\end{abstract}


\maketitle

\section{Introduction}

In the present paper we consider overdetermined problems for the fractional Laplacian in unbounded exterior sets or bounded annular sets. 
Different cases will be taken into account,
but the results obtained will lead in any case to 
the classification of the solution and of
the domain, that will be shown to possess rotational symmetry.

The notation used in this paper will be the following.
Given an open set whose boundary is of class~$\mathcal{C}^2$,
we denote by~$\nu$ the inner unit normal vector and
for any~$x_0$ on the boundary of such set, we use the notation
\begin{equation}\label{def: s-derivative}
(\pa_{\nu})_s u(x_0) := \lim_{t \to 0^+} \frac{u(x_0 + t \nu(x_0))-u(x_0)}{t^s}.
\end{equation}
Of course, when writing such limit, we always assume that the limit
indeed exists and we call the above quantity
the \emph{inner normal $s$-derivative} of $u$ at~$x_0$.
The parameter~$s$ above ranges in the interval~$(0,1)$
and it corresponds to the fractional Laplacian~$(-\Delta)^s$.
A brief summary of the basic theory of the fractional Laplacian will
be provided in Section~\ref{sec: preliminaries}: for the moment,
we just remark that~$(-\Delta)^s$ reduces to (minus) the classical
Laplacian as~$s\to 1$, and the quantity in~\eqref{def: s-derivative}
becomes in this case the classical Neumann condition along
the boundary.\medskip

With this setting, we are ready to state our results.
For the sake of clarity, we first give some simplified versions
of our results which are ``easy to read'' and deal with ``concrete''
situations. These results will be indeed just particular
cases of more general theorems that will be presented
later in Section~\ref{FFGG}.

More precisely, the results we present are basically of two types.
The first type deals with \emph{exterior sets}.
In this case, the equation is assumed to
hold in~$\R^N \setminus \overline{G}$,
where~$G$
is a non-empty open bounded set
of $\R^N$, not necessarily connected, whose boundary is
of class~$\mathcal{C}^2$, and~$\overline{G}$ denotes the closure of $G$. 
We sometimes split $G$ into its connected components by writing
\[
G= \bigcup_{i=1}^k G_i,
\]
where any $G_i$ is a bounded, open and connected set
of class~$\mathcal{C}^2$, and, to avoid pathological situations, we suppose that $k$ is finite. Notice that $\overline{G_i} \cap \overline{G_j} = \emptyset$ if $i \neq j$.
\medskip

Then, the prototype of our results for exterior sets is the following.

\begin{theorem}\label{thm: example}
Let us assume that there exists $u \in \mathcal{C}^s(\R^N)$ such that 
\[
\begin{cases}
(-\Delta)^s u =0 & \text{in $\R^N \setminus \overline{G}$} \\
u = a>0  & \text{in $\overline{G}$} \\
u(x) \to 0 & \text{as $|x| \to +\infty$} \\
(\pa_{\nu})_s u= const.= \alpha_i \in \R & \text{on $\pa G_i$}.
\end{cases} 
\]
Then $G$ is a ball, and $u$ is radially symmetric and radially decreasing with respect to the centre of $G$.
\end{theorem}

The second case dealt with in this paper is the one
of \emph{annular sets}. Namely, in this case the equation is supposed
to hold in~$\Omega \setminus \overline{G}$,
where~$\Omega$ is a bounded set that contains~$G$
and such that~$\Omega \setminus \overline{G}$ is of class $\mathcal{C}^2$. 
\medskip

Then, the prototype of our results for annular sets is the following.

\begin{theorem}\label{thm: example 2}
Let us assume that there exists $u \in \mathcal{C}^s(\R^N)$ such that 
\[
\begin{cases}
(-\Delta)^s u =0 & \text{in $\Omega \setminus \overline{G}$} \\
u = a>0  & \text{in $\overline{G}$} \\
u= 0 & \text{in $\R^N \setminus \overline{\Omega}$} \\
(\pa_{\nu})_s u= const.= \alpha_i \in \R & \text{on $\pa G_i$} \\
(\pa_{\nu})_s u= const.= \beta \in \R & \text{on $\pa \Omega$},
\end{cases} 
\]
where $(\pa_\nu)_s u$ denotes the inner (with respect to $\Omega \setminus \overline{G}$) normal $s$-derivative of $u$. Then $G$ and $\Omega$ are concentric balls, and $u$ is radially symmetric and radially decreasing with respect to the centre of $G$.
\end{theorem}

We stress that, here and in the following, we do not assume that $\R^N \setminus \overline{G}$ (resp. $\Omega \setminus \overline{G}$) is connected; this is why we do not use the common terminology \emph{exterior domain} (resp. \emph{annular domain}), which indeed refers usually
to an exterior connected set (resp. annular connected set). 

We also stress that, while the Dirichlet boundary datum has to be the same for all the connected components of $G$, the Neumann boundary datum can vary.
\medskip

Overdetermined elliptic problems have a long history, which begins with the seminal paper by J. Serrin \cite{Serrin}. A complete review of the results which have been obtained since then goes beyond the aims of this work. In what follows, we only review the contributions regarding exterior or annular domains in the local case, and the few results which are available about overdetermined problems for the fractional Laplacian. 

Overdetermined problems for the standard Laplacian in exterior domains have been firstly studied by W. Reichel in \cite{Reichel1}, where he assumed that both $G$ and $\R^N \setminus \overline{G}$ are connected. In such a situation, W. Reichel proved that if there exists $u \in \mathcal{C}^2(\R^N \setminus G)$ (i.e. of class $\mathcal{C}^2$ up to the boundary of the exterior domain) such that 
\begin{equation}\label{pb: overdet local}
\begin{cases}
-\Delta u =f(u) & \text{in $\R^N \setminus \overline{G}$} \\
u = a>0  & \text{on $\pa G$} \\
u(x) \to 0 & \text{if $|x| \to +\infty$}\\
\pa_{\nu} u= const.= \alpha \le 0 & \text{on $\pa G$} \\
0 \le u <a & \text{in $\R^N \setminus \overline{G}$},
\end{cases} 
\end{equation}
where $\pa_\nu$ denotes the usual normal derivative, and $f(t)$ is a locally Lipschitz function, non-increasing for non-negative and small values of $t$, then $G$ has to be a ball, and $u$ is radially symmetric and radially decreasing with respect to the centre of $G$. The proof is based upon the moving planes method. 

With a different approach, based upon the moving spheres, A. Aftalion and J. Busca \cite{AftBus} addressed the same problem when $f$ is not necessarily non-increasing for small positive values of its argument. In particular, they could treat the interesting case $f(t) = t^p$ for $N/(N-2) < p \le (N+2)/(N-2)$. 

Afterwards, B. Sirakov \cite{Sirakov} proved that the
result obtained in~\cite{Reichel1}
holds without the assumption $u<a$, and for possibly multi-connected sets $G$. Moreover, he allowed different boundary conditions on the different components of $G$: that is, $u=a>0$ and $\pa_\nu u = \alpha \le 0$ on $\pa G$ can be replaced by $u=a_i>0$ and $\pa_\nu u=\alpha_i \le 0$ on $\pa G_i$, with $a_i$ and $\alpha_i$ depending on $i=1,\dots,k$. His method works also in the setting considered in \cite{AftBus}.

%% We point out that, in the previous contributions, the nonlinearity $f$ can depend also on the gradient of $u$. We prefer to omit such dependence to put in evidence the similarities and the differences with our results. 

We point out that in \cite{Sirakov} a quasi-linear regular strongly elliptic operator has been considered instead of the Laplacian. Concerning quasi-linear but possibly degenerate operators, we refer to \cite{Reichel2}.

As far as overdetermined problems in annular domains is concerned, we refer the reader to \cite{Aless, Philippin, PayPhil, Reichel3}. In \cite{Aless}, G. Alessandrini proved the local counterpart of Theorem \ref{thm: example 2} for quasi-linear, possibly degenerate, operators. This enhanced the results in \cite{Philippin, PayPhil}. In \cite{Reichel3}, W. Reichel considered inhomogeneous equations for the Laplace operator in domains with one cavity.

Regarding the nonlocal framework, the natural counterpart of the J. Serrin's problem for the fractional $s$-Laplacian ($0<s<1$) has been recently studied by M. M. Fall and S. Jarohs in \cite{FallJarohs}. In such contribution the authors introduced the main tools for dealing with nonlocal overdetermined problems, such as comparison principles for anti-symmetric functions and a nonlocal version of the Serrin's corner lemma. Such results will be used in our work. We refer also to \cite{Dalibard}, where the authors considered a similar problem in dimension $N=2$, and for $s=1/2$. In both the quoted papers the $s$-normal derivative defined in \eqref{def: s-derivative} plays the role of the normal derivative in the local setting. This is motivated by the regularity theory developed in \cite{SeR-O}, see also \cite{Grubb}, where the corresponding ``$s$-Neumann boundary value problem" is studied.
\medskip

As already mentioned, Theorems~\ref{thm: example}
and~\ref{thm: example 2}
are just simplified versions of more general results
that we obtained. Next section will present the results
obtained in full generality.

\subsection{The general setting}\label{FFGG}

Now we present our results in full generality.
For this, first we consider the
overdetermined problem in an exterior set~$\R^N \setminus \overline{G}$, namely:
\begin{equation}\label{pb: overdet}
\begin{cases}
(-\Delta)^s u =f(u) & \text{in $\R^N \setminus \overline{G}$} \\
u = a>0  & \text{in $\overline{G}$} \\
(\pa_{\nu})_s u= const.= \alpha_i \in \R & \text{on $\pa G_i$.} \\
\end{cases}
\end{equation}
Recall that $\R^N \setminus \overline{G}$ will always be assumed to be of class $\mathcal{C}^2$, and note that the $s$-Neumann boundary datum can depend on $i$. Our first main result in this framework is the following.

\begin{theorem}\label{thm: main 1}
Let~$d(x)$ to be the distance of~$x\in\overline\Omega$
to the boundary of~$\Omega$, and let us assume that
\begin{equation}\label{W buona intro}
 \frac{u}{d^s} \in {\mathcal{C}}^{0,1}(\overline\Omega).
\end{equation}
Let $f(t)$ be a locally Lipschitz function, non-increasing for nonnegative and small values of $t$, and let $\alpha_i <0$ for every $i$. If there exists a weak solution $u \in \mathcal{C}^s(\R^N)$ of \eqref{pb: overdet}, satisfying
\[
\text{$0 \le u < a$ in $\R^N \setminus \overline{G}$ and $u(x) \to 0$ as $|x| \to +\infty$},
\]
then $G$ is a ball, and $u$ is radially symmetric and radially decreasing with respect to the centre of $G$. 
\end{theorem}

\begin{remark}
Assumption \eqref{W buona intro} is a regularity assumption on the problem. Having assumed that $f$ is locally Lipschitz and that $u \in \mathcal{C}^s(\R^N)$ is bounded, \eqref{W buona intro} is satisfied for instance if $\Omega$ is of class $C^\infty$, see \cite{Grubb}, or if $\Omega$ is $\mathcal{C}^{2,\alpha}$ for some $\alpha \in (0,1)$ and $s >1/2$, see \cite{BRR}.
\end{remark}

The concept of weak solution and its basic properties
will be recalled in Section~\ref{sec: preliminaries}
for the facility of the reader.

We point out that under additional assumptions on $f$, both the assumptions \eqref{W buona intro} and $\alpha_i < 0$ in Theorem~\ref{thm: main 1}
can be dropped, and the condition $0 \le u < a$ can be relaxed.

\begin{theorem}\label{thm: main 1 prime}
Let $f$ be a locally Lipschitz function, non-increasing in the whole interval $[0,a]$. If there exists a weak solution $u \in \mathcal{C}^s(\R^N)$ of \eqref{pb: overdet} satisfying
\[
\text{$0 \le u \le a$ in $\R^N  \setminus \overline{G}$ and $u(x) \to 0$ as $|x| \to +\infty$},
\]
then $G$ is a ball, and $u$ is radially symmetric and radially decreasing with respect to the centre of $G$. 
\end{theorem}

Other results in the same direction are the following.

\begin{corollary}\label{corol: main 1}
Under the assumptions of Theorem \ref{thm: main 1}, let us assume that $f(a) \le 0$; then the request $\alpha_i<0$ in
the statement of
Theorem~\ref{thm: main 1}
is not necessary, and condition $0 \le u <a$ can be replaced by $0 \le u \le a$. 
\end{corollary}

\begin{corollary}\label{corol: subharmonicity}
If under the assumptions of Theorem \ref{thm: main 1}, we suppose that $f(t) \le 0$ for $t \ge 0$, then the request $\alpha_i<0$ is not necessary, and condition $0 \le u <a$ can be replaced by $u \ge 0$. 
Analogously, if under the assumptions of Theorem \ref{thm: main 1 prime}, we suppose that $f(t) \le 0$ for $t \ge 0$, then condition $0 \le u \le a$ can be replaced by $u \ge 0$. 
\end{corollary}

The proof of the corollaries is based upon simple comparison arguments. When $f \ge 0$, in the same way one could show that assumption $u \ge 0$ is not necessary, obtaining in particular Theorem \ref{thm: example}.

It is worth to notice that
the regularity assumption $u \in \mathcal{C}^s(\R^N)$
is natural in our framework, see Theorem \ref{thm: regularity} in the appendix at the end of the paper: what we mean is that each bounded weak solution of the first two equations in \eqref{pb: overdet} is of class $\mathcal{C}^s(\R^N)$. All the previous results (and the forthcoming ones) could have stated for bounded weak solutions, without any regularity assumption. We preferred to assume from the beginning that $u \in \mathcal{C}^s(\R^N)$, since in this way the condition on $(\pa_\nu)_s u$ makes immediately sense, without further observations. The $\mathcal{C}^s(\R^N)$ regularity is optimal, as shown by the simple 
example
\[
\begin{cases}
(-\Delta)^s u=1 & \text{in $B_1$}\\
u=0 & \text{in $\R^N \setminus B_1$},
\end{cases}
\]
which has the explicit solution $u(x)= \gamma_{N,s} (1-|x|^2)^s_+$, where $v_+$ denotes the positive part of $v$, and $\gamma_{N,s}$ is a normalization constant depending on $N$ and on $s$. 
Also, eigenfunctions are not better than ${\mathcal{C}}^s(\R^N)$,
see e.g.~\cite{MR3233760}.

\begin{remark}\label{rem: on regularity}
In the local setting \cite{Reichel1}, see also \cite{Sirakov}, the solution $u$ of \eqref{pb: overdet local} is supposed to be of class $\mathcal{C}^2$ up to the boundary. In this way, the authors could avoid to assume that $\alpha <0$: indeed, in Proposition 1 in \cite{Reichel1}, as well as in Step 4 in the proof of the main results in \cite{Sirakov}, the authors computed the second derivatives of $u$ on the boundary of $G$. In our context, it seems not natural to ask that $u$ has better regularity than $\mathcal{C}^s$, and this is the main reason for which we
need to suppose $\alpha<0$ in Theorem~\ref{thm: main 1}. 
\end{remark}

\begin{remark}
We think that it is worth to point out that there exist solutions of \eqref{pb: overdet} satisfying all the assumptions of the above statements when $G$ is a ball. Such existence results will be stated and proved 
at the end of the paper, in Section~\ref{sec: existence}, for the sake of completeness.
\end{remark}

\begin{remark}
As already pointed out, the fact that $\R^N \setminus \overline{G}$ is not supposed to be connected marks a difference with respect to the local case. The same difference, which arises also in \cite{FallJarohs}, is related to the non-local nature of both the fractional Laplacian and the boundary conditions.
\end{remark}

Now we present our results in the general form for
an annular set $\Omega \setminus \overline{G}$
(recall that in this case $\Omega$ is bounded,
and, by the regularity assumed,
$G$ cannot be internally tangent to $\partial \Omega$). {Recall that $\Omega$ can be multi-connected but, in this case, we assume that $\Omega$ has a finite number of connected components.}
The natural counterpart of the overdetermined
problem \eqref{pb: overdet} for annular sets is given by
\begin{equation}\label{pb: overdet bounded}
\begin{cases}
(-\Delta)^s u =f(u) & \text{in $\Omega \setminus \overline{G}$} \\
u = a>0  & \text{in $\overline{G}$} \\
u = 0 & \text{in $\R^N \setminus \overline{\Omega}$} \\
(\pa_{\nu})_s u= \alpha_i \in \R & \text{on $\pa G$} \\
(\pa_{\nu})_s u= \beta \in \R & \text{on $\pa \Omega$},
\end{cases}
\end{equation}
where the notation $(\pa_{\nu})_s$ is used for the inner normal $s$-derivative in $\pa (\Omega \setminus \overline{G})$. 
In this setting we have:

\begin{theorem}\label{thm: main 2}
Let $f$ be locally Lipschitz. Let us assume that there exists a weak solution $u \in \mathcal{C}^s(\R^N)$ of \eqref{pb: overdet} satisfying
\[
0 < u < a \qquad \text{in $\Omega \setminus \overline{G}$}.
\]
Then $\Omega$ and $G$ are concentric balls, and $u$ is radially symmetric and radially decreasing with respect to their centre. 
\end{theorem}

Note that in this case no assumption on the monotonicity of $f$, or on the sign of $\alpha_i$ and $\beta$, is needed. Moreover, the result holds for $\Omega \setminus \overline{G}$ of class $\mathcal{C}^2$, without any further assumption such as \eqref{W buona intro}. 

As for the problem in exterior sets, the condition $0 < u <a$ can be relaxed under additional assumptions on $f$.

\begin{theorem}\label{thm: main 2 prime}
Let $f$ be a locally Lipschitz function, non-increasing in $[0,a]$. If there exists a weak solution $u \in \mathcal{C}^s(\R^N)$ of \eqref{pb: overdet} satisfying
\[
\text{$0 \le u \le a$ in $\R^N$},
\]
then both $\Omega$ and $G$ are concentric balls, and $u$ is radially symmetric and radially decreasing with respect to their centre. 
\end{theorem}

\begin{corollary}\label{corol: main 2}
Under the assumptions of Theorem \ref{thm: main 2}, let us assume that $f(a) \le 0$; then the condition $0 < u <a$ can be replaced by $0 < u \le a$ in $\R^N$. Analogously, if $f(0) \ge 0$, then the condition $0 < u <a$ can be replaced by $0 \le u < a$ in $\R^N$. 
\end{corollary}

Clearly,
if both $f(a) \le 0$ and $f(0) \ge 0$, we obtain the thesis for $\alpha_i,\beta \in \R$ and $0 \le u \le a$,
and then we also obtain
Theorem~\ref{thm: example 2}
as a particular case.

The proof of Corollary \ref{corol: main 2} is analogue to that of Corollary \ref{corol: main 1}
(and thus will be omitted). One could also state a counterpart of Corollary \ref{corol: subharmonicity} in the present setting.\medskip

At last, we observe that when $G$ (or $\Omega \setminus \overline{G}$) is a priori supposed to be radial, our method permits to deduce the radial symmetry of the solutions of the Dirichlet problem. In the local framework, this type of results have been proved in \cite{Reichel3,Reichel1,Sirakov}. 

\begin{theorem}\label{thm: radial 1}
Let $B_{\rho}(x_0)$ be a ball, and let $u \in \mathcal{C}^s(\R^N)$ be a weak solution of
\[
\begin{cases}
(-\Delta)^s u= f(u) & \text{in $\R^N \setminus \overline{B_\rho(x_0)}$} \\
u = a>0  & \text{in $\overline{B_\rho(x_0)}$},
\end{cases}
\]
such that 
\begin{equation}\label{cond Neumann}
(\pa_\nu)_s u <0 \quad \text{on $\pa B_{\rho}(x_0)$}.
\end{equation}
If $f(t)$ is a locally Lipschitz function, non-increasing for nonnegative and small values of $t$, and $u$ satisfies
\[
\text{$0 \le u < a$ in $\R^N \setminus \overline{B_{\rho}(x_0)}$ and $u(x) \to 0$ as $|x| \to +\infty$},
\]
then $u$ is radially symmetric and radially decreasing with respect to $x_0$. 
\end{theorem}

When compared with the local results, condition \eqref{cond Neumann} seems not to be natural. On the other hand, for the reasons already explained in Remark \ref{rem: on regularity} we could not omit it in general. Nevertheless, under additional assumptions on $f$ it can be dropped. 

\begin{corollary}\label{corol: radial}
If under the assumptions of Theorem \ref{thm: radial 1} we suppose that $f(a) \le 0$, then \eqref{cond Neumann} can be omitted, and condition $0 \le u<a$ can be replaced by $0 \le u \le a$. If moreover $f(t) \le 0$ for every $t \le 0$, then condition $0 \le u <a$ can be replaced by $u \ge 0$.
\end{corollary}

It is clear that similar symmetry results hold for Dirichlet problems in annuli. 

An interesting limit case takes place when $G=\{x_0\}$ is a single point of $\R^N$. The reader can easily check that the proof of Theorem \ref{thm: radial 1} works also in this setting. With some extra work, we can actually obtain a better result. %Without loss of generality, we can assume that $x_0=0$.

\begin{theorem}\label{thm: radial point}
Let us assume that there exists a bounded weak solution of 
\[
\begin{cases}
(-\Delta)^s u=f(u) & \text{in $\R^N \setminus \{x_0\}$} \\
u(x_0)=a.
\end{cases}
\]
If $f(t)$ is a locally Lipschitz functions, non-increasing for nonnegative and small values of $t$, and $u$ satisfies
\[
\text{$0 \le u \le a$ in $\R^N \setminus \{x_0\}$ and $u(x) \to 0$ as $|x| \to +\infty$},
\]
then $u$ is radially symmetric and radially decreasing with respect to $x_0$. 
\end{theorem}

Notice that no condition on the $s$-normal derivative is needed. For this reason, we omitted the assumption $u \in \mathcal{C}^s(\R^N)$, which anyway, by Theorem \ref{thm: regularity}, would be natural also in this context. As a straightforward corollary, we obtain a variant of the Gidas-Ni-Nirenberg symmetry result for the fractional Laplacian.

\begin{corollary}
Let $u$ be a nonnegative bounded weak solution of $(-\Delta)^s u=f(u)$ in $\R^N$.
If $f(t)$ is a locally Lipschitz functions, non-increasing for nonnegative and small values of $t$, and $u$ satisfies
\[
\text{$u(x) \to 0$ as $|x| \to +\infty$},
\]
then $u$ is radially symmetric and radially decreasing with respect to a point of $\R^N$.
\end{corollary}

Analogue symmetry results have been proved in \cite{FelmerWang}, for a different class of nonlinearities $f$ (having non-empty intersection with the one considered here). We point out that, with respect to \cite{FelmerWang}, we do not require any condition at infinity on the decay of $u$.

\subsection{Outline of the paper} The basic technical definitions
needed in this paper will be recalled in Section~\ref{sec: preliminaries}.

In Section \ref{sec: over exterior}
we consider overdetermined problems in exterior sets, proving Theorems \ref{thm: main 1}, \ref{thm: main 1 prime} and Corollaries \ref{corol: main 1}, \ref{corol: subharmonicity}. Section \ref{sec: over annular} is devoted to overdetermined problems in annular sets. In Section \ref{sec: radial} we study the symmetry of the solutions when the domain is a priori supposed to be radial. In Section \ref{sec: existence} we present some existence results. Finally, in a brief appendix we discuss the regularity of bounded weak solutions of Dirichlet fractional problems in unbounded sets.

\medskip

\noindent \textbf{Acknowledgements.} We wish to thank Prof. Boyan Sirakov for useful discussions on his paper \cite{Sirakov}. Part of this work was carried out when the first author was visiting the Weiestrass Institut f\"ur Angewandte Analysis und Stochastik in Berlin, and he wishes to thank for the hospitality.

\section{Definitions and preliminaries}\label{sec: preliminaries}

We collect in this section some definitions and results which will be used in the proofs of the main theorems. 

\subsection{Definitions} Let $N \ge 1$ and $s \in (0,1)$. For a function $u \in \mathcal{C}^\infty_c(\R^N)$, the fractional $s$-Laplacian is defined by
\begin{equation}\label{integral representation}
\begin{split}
(-\Delta)^s u(x) :&= c_{N,s} \pv \int_{\R^N} \frac{u(x)-u(y)}{|x-y|^{N+2s}}\,dy \\
& = c_{N,s} \lim_{\eps \to 0^+}  \int_{ \{ |y-x| > \eps \} } \frac{u(x)-u(y)}{|x-y|^{N+2s}} \, dy ,
\end{split}
\end{equation}
where $c_{N,s}$ is a normalization constant, and $\pv$ stays for ``principal value". In the rest of the paper, to simplify the notation we will always omit both $c_{N,s}$ and $\pv$. The bilinear form associated to the fractional Laplacian is
\[
\mathcal{E}(u,v) := \frac{c_{N,s}}{2} \int_{\R^{2N}} \frac{(u(x)-u(y)) (v(x)-v(y) )}{|x-y|^{N+2s}}\,dy.
\]
It can be proved that $\mathcal{E}$ defines a scalar product, and we denote by $\mathcal{D}^s(\R^N)$ the completion of $\mathcal{C}^\infty_c(\R^N)$ with respect to the norm induced by $\mathcal{E}$. We also introduce, for an arbitrary open set $\Omega \subset \R^N$, the space $H^s(\Omega):= L^2(\Omega) \cap \mathcal{D}^s(\R^N)$. It is a Hilbert space with respect to the scalar product $\mathcal{E}(u,v) + \langle u,v \rangle_{L^2(\Omega)}$, where $\langle \cdot, \cdot \rangle_{L^2(\Omega)}$ stays for the scalar product in $L^2(\Omega)$. The case $\Omega=\R^N$ is admissible. We write that $u \in H^s_{\loc}(\R^N)$ if $u \in H^s(K)$ for every compact set $K \subset \R^N$.
 
A function $w$ is a weak supersolution of
\[
(-\Delta)^s w \ge  g(x)  \qquad \text{in $\Omega$},
\]
if $w  \in \mathcal{D}^s(\R^N)$ and
\begin{equation}\label{weak sol}
\mathcal{E}(w,\varphi) \ge \int_{\Omega} g(x) \varphi(x) \, dx \qquad \forall \varphi \in \mathcal{C}^\infty_c(\Omega), \ \varphi \ge 0.
\end{equation}
If the opposite inequality holds, we write that $w$ is a weak subsolution. If $w \in \mathcal{D}^s(\R^N)$ and equality holds in \eqref{weak sol} for every $\varphi \in \mathcal{C}^\infty_c(\Omega)$, then we write that $w$ is a weak solution of 
\begin{equation}\label{vis sol}
(-\Delta)^s w =  g(x)  \qquad \text{in $\Omega$}.
\end{equation}
Since we will always consider weak solutions (supersolutions, subsolutions), the adjective weak will sometimes be omitted.

\subsection{Regularity results}\label{sub: regularity}
Let $u \in L^\infty(\R^N) \cap \mathcal{D}^s(\R^N)$ be a weak solution of
\begin{equation}\label{pb: Dirichlet}
\begin{cases}
(-\Delta)^s u= f(u) & \text{in $\R^N \setminus \overline{G}$} \\
u=a & \text{in $\overline{G}$},
\end{cases}
\end{equation}
with $f$ locally Lipschitz continuous. By Theorem \ref{thm: regularity} in the appendix, we know that
\begin{itemize}
\item $u \in \mathcal{C}^{1,\sigma}(\R^N \setminus \overline{G})$ for some $\sigma \in (0,1)$ (\emph{interior regularity}), and the integral representation \eqref{integral representation} holds point-wise;
\item $u \in \mathcal{C}^s(\R^N)$, and in particular $u/\delta^s \in \mathcal{C}^{0,\gamma}(\R^N \setminus G)$ for some $\gamma \in (0,1)$, where $\delta$ denotes the distance from the boundary of $G$ (\emph{boundary regularity}).
\end{itemize}
Since any weak solution $u \in \mathcal{C}^s(\R^N)$ of \eqref{pb: Dirichlet} such that $u(x) \to 0$ as $|x| \to +\infty$ is in $L^\infty(\R^N)$ (and also in $\mathcal{D}^s(\R^N)$, by definition of weak solution), the previous regularity results will be used throughout the rest of the paper.

\subsection{Comparison principles}

We recall a
strong maximum principle
and a Hopf's lemma for anti-symmetric functions \cite[Proposition 3.3 and Corollary 3.4]{FallJarohs}. In the quoted paper, the strong maximum principle is stated under the assumption that $\Omega$ is bounded, but for the proof this is not necessary. As a result, the following holds.

\begin{proposition}[Fall, Jarohs \cite{FallJarohs}]\label{STRONG}
Let $H \subset \R^n$ be a half-space, and let $\Omega \subset H$ (not necessarily bounded). Let $c \in L^\infty(\Omega)$, and let $w$ satisfy
\[
\begin{cases}
(-\Delta )^s w + c(x) w \ge 0 & \text{in $\Omega$} \\
w(x) = -w(\bar x) & \text{in $\R^N$} \\
w \ge 0 & \text{in $H$},
\end{cases}
\]
where $\bar x$ denotes the reflection of $x$ with respect to $\partial H$. Then either $w>0$ in $\Omega$, or $w \equiv 0$ in $H$. Furthermore, if $x_0 \in \partial \Omega \setminus \partial H$ and $w(x_0) = 0$, then $(\pa_\eta)_s w(x_0)<0$, where $\eta$ is the outer unit normal vector of $\Omega$ in $x_0$. 
\end{proposition}

\begin{remark}
If $\Omega \subset H$ shares part of its boundary with the hyperplane $\pa H$, and $x_0 \in \partial H \cap \pa \Omega$, then we cannot apply the Hopf's lemma in $x_0$, since it is necessary to suppose that it lies on the boundary of a ball compactly contained in $H$. This assumption is used in the proof in \cite{FallJarohs}.
\end{remark}

In the proof of Theorem \ref{thm: main 1}, we shall need a version of the Hopf's lemma allowing to deal with points of $\partial \Omega \cap \partial H$. To be more precise, let $\Omega'$ be a $\mathcal{C}^2$ set in $\R^N$, symmetric with respect to the hyperplane $T$, and let $H$ be a half-space such that $T=\partial H$. Let $\Omega := H \cap \Omega'$, and let us assume that $w \in \mathcal{C}^s(\R^N)$ satisfies
\[
\begin{cases}
(-\Delta)^s w + c(x) w = 0 & \text{in $\Omega'$} \\
w(x) = -w(\bar x) \\
w > 0 & \text{in $\Omega$} \\
w \ge 0 & \text{in $H$},
\end{cases}
\]  
where $c \in L^\infty(\Omega')$ and $\bar x$ denotes the reflection of $x$ with respect to $T$. We note that $\partial \Omega \cap T$ is divided into two parts: a regular part of Hausdorff dimension $N-1$, which is a relatively open set in $T$, and a singular part of Hausdorff dimension $N-2$, which is $\partial \Omega' \cap T$. We also note that, by anti-symmetry, $w(x)=0$ for every $x \in T$.

\begin{proposition}\label{prop: new hopf}
In the previous setting, if $x_0$ is a point in the regular part of $\pa \Omega \cap T$, then
\begin{equation}\label{1st}
-\liminf_{t \to 0^+} \frac{w(x_0-t\nu(x_0))}{t} <0,
\end{equation}
where $\nu(x_0)$ is the outer unit normal vector to $T=\partial H$ in $x_0$.
\end{proposition}

\begin{remark}
At a first glance it could be surprising that a boundary lemma involving a fractional problem gives a result on the full outer normal derivative of the function. 
Namely, functions satisfying a fractional equation of order~$2s$
are usually not better than~$\mathcal{C}^s(\R^N)$ at the boundary, so
the first order incremental quotient in~\eqref{1st} is in general
out of control.
But in our case, if we look at the picture more carefully, we realize that since we are assuming that $w$ is an anti-symmetric solution in the whole $\Omega'$ (which contains both $\Omega$ and its reflection), any $x_0$ on the regular part of $\pa \Omega \cap T$ is actually an interior point for $w$, and hence it is natural to expect some extra regularity.
\end{remark}

\begin{proof}[Proof of Proposition~\ref{prop: new hopf}]
Without loss of generality, we can assume that 
\[
T=\{x_N=0\} \quad \text{and} \quad H=\{x_N >0\}, 
\]
so that $\Omega=\Omega' \cap \{x_N >0\}$. Let $\rho>0$ be such that $B_{\rho}(x_0) \Subset \Omega'$. If necessary replacing $\rho$ with a smaller quantity, we can suppose that $B:=B_{\rho}(x_0',4\rho)$ and $B':=B_{\rho}(x_0',-4\rho)$ are both compactly contained in $\Omega'$. Now we follow the strategy of Lemma 4.4 in \cite{FallJarohs}: for $\alpha >0$ to be determined in the sequel, we consider the barrier
\[
h(x):= x_N \left( \varphi(x) + \alpha(d_1(x) +d_2(x) ) \right),
\]
where 
\[
\varphi(x) = (\rho^2 - |x-x_0|^2)_+^s
\]
is the positive solution of 
\[
\begin{cases}
(-\Delta)^s \varphi = 1 & \text{in $B_1$} \\
\varphi = 0 & \text{in $\R^N \setminus \overline{B_1}$},
\end{cases}
\]
and 
\[
d_1(x) := (\rho-|x-(x_0', 4\rho)|)_+ \qquad d_2(x) := (\rho-|x-(x_0', -4\rho)|)_+
\]
are the truncated distance functions from the boundary of $B$ and $B'$, respectively. With this definition, we can compute $(-\Delta)^s h$ exactly as in Lemma 4.4 in \cite{FallJarohs}, proving that
\[
(-\Delta)^s h(x) - c(x) h(x) \le (C_1-\alpha C_2) |x_N| \le 0 \qquad \text{for every $x \in B_{\rho(x_0) \cap \{x_N >0\}}$},
\]
provided $\alpha>0$ is sufficiently large. By continuity and recalling that $w>0$ in $\Omega$, we deduce that $w \ge C >0$ in $B$. This permits to choose a positive constant $\sigma >0$ such that $w \ge \sigma h$ in $\overline{B_1}$, and hence $w-\sigma h \ge 0$ in $\Omega \setminus B_{\rho}(x_0)$. Therefore, the weak maximum principle (Proposition 3.1 in \cite{FallJarohs}) implies that $w \ge \sigma h$ in $\Omega$, and in particular
\[
w(x_0',t) \ge \sigma t (\rho^2 -t^2)_+^s \qquad \forall t \in [0,\rho),
\]
which gives the desired result. 
%Notice that, by regularity, $w$ is of class $\mathcal{C}^1$ in $B_\rho(x_0)$, and hence the pointwise estimate makes sense.
\end{proof}

As far as overdetermined problem in bounded exterior sets, namely
Theorem \ref{thm: main 2}, we shall make use of a maximum principle in domain of small measure, proved in \cite[Proposition 2.4]{JarohsWeth} in a parabolic setting. In our context, the result reads as follows.

\begin{proposition}[Jarohs, Weth \cite{JarohsWeth}]\label{SMALL}
Let $H \subset \R^N$ be a half-space and let $c_\infty>0$. There exists $\delta=\delta(N,s,c_\infty)>0$ such that if $U \subset H$ with $|U| < \delta$, and $u$ satisfies in a weak sense 
\[
\begin{cases}
-\Delta w +c(x) w \ge 0 & \text{in $U$} \\
w \ge 0 & \text{in $H \setminus U$} \\
w(x) = -w(\bar x),
\end{cases} 
\]
with $\|c\|_{L^\infty(U)} < c_\infty$, then $w \ge 0$ in $U$.
\end{proposition}

\section{The overdetermined problem in exterior sets}\label{sec: over exterior}

In the first part of this section, we prove Theorem \ref{thm: main 1}.

We follow the same sketch used by Reichel in \cite{Reichel1}, applying the moving planes method to show that for any direction $e \in \mathbb{S}^{N-1}$ there exists $\bar \lambda=\bar \lambda(e)$ such that both the $G$ and the solution $u$ are symmetric with respect to a hyperplane
\[
T_{\lambda}:= \left\{ x \in \R^N: \langle x, e \rangle = \lambda\right\},
\]
where $\langle\cdot,\cdot \rangle$ denotes the Euclidean scalar product in $\R^N$. In the following we fix the direction $e=e_N$ and use the notation $x=(x',x_N) \in \R^{N-1} \times \R$ for points of $\R^N$. For $\lambda \in \R$, we set
\begin{equation}\label{notation}
\begin{split}
T_\lambda & := \{x \in \R^N: x_N=\lambda\}; \\
H_{\lambda} & := \{x \in \R^N: x_N>\lambda\}; \\
x^\lambda & := (x',2\lambda-x_N) \quad \text{the reflection of $x$ with respect to $T_\lambda$}; \\
A^\lambda&:= \text{the reflection of a given set $A$ with respect to $T_\lambda$}; \\
\Sigma_\lambda &:= H_\lambda \setminus \overline{G^\lambda} \quad \text{the so-called \emph{reduced half-space}}; \\
d_i &:= \inf \left\{ \lambda \in \R: T_\mu \cap \overline{G_i} = \emptyset \quad \text{for every $\mu > \lambda$}\right\}.
\end{split}
\end{equation}

\begin{figure}[ht]
\includegraphics[height=5cm]{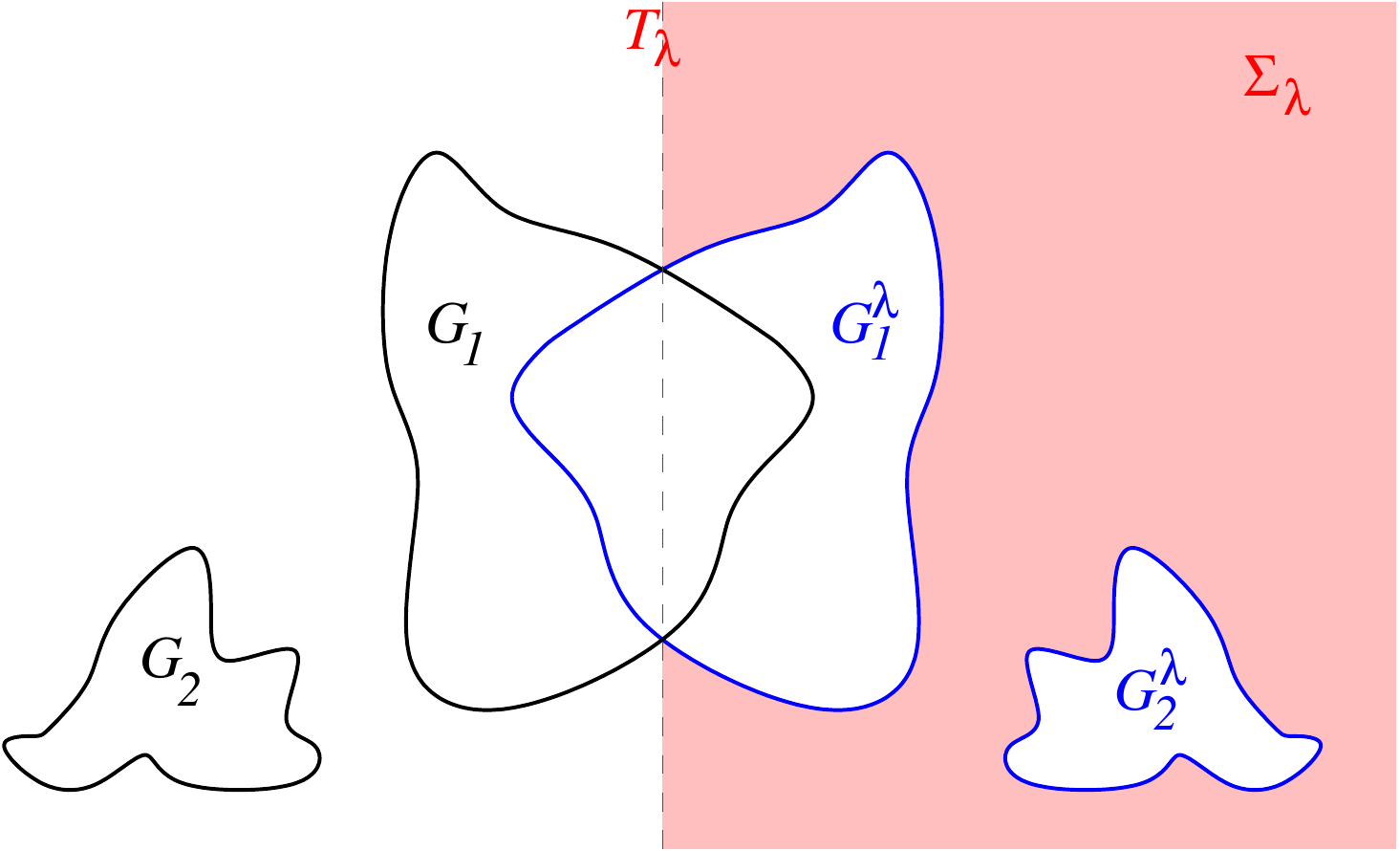} 
\caption{A picture of the reflected sets and the reduced half-space.}
\end{figure}

It is known 
that, for $\lambda$ a little smaller than $d_i$,
the reflection of $G_i \cap H_\lambda$ with respect to $T_\lambda$
lies inside $G_i$, namely
\begin{equation}\label{lies i}
(\overline{G_i} \cap H_\lambda)^\lambda \subset G_i \cap H_\lambda^\lambda \quad \text{with strict inclusion}.
\end{equation}
In addition, $\langle \nu(x) ,e_N \rangle > 0$ for every $x \in \pa G_i \cap T_\lambda$ (we recall that $\nu$ denotes the outer normal vector on $\pa G_i$, thus directed directed towards the interior of $\R^N \setminus \overline{G_i}$); this remains true for decreasing values of $\lambda$ up to a limiting position $\bar \lambda_i$ such that one of the following alternatives takes place:
\begin{itemize}
\item[($i$)] \emph{internal tangency}: the reflection $(\overline{G_i} \cap H_\lambda)^\lambda$ becomes internally tangent to $\pa G_i$;
\item[($ii$)] \emph{orthogonality condition}: $\langle \nu(x_0) ,e_N \rangle = 0$ for some $x_0 \in \pa G_i \cap T_\lambda$.
\end{itemize}

Let $\bar \lambda:= \max\{\bar \lambda_i: i=1,\dots,k\}$. For $\lambda > \bar \lambda$, it is clear that inclusion \eqref{lies i} holds for every $i$. Since $\overline{G_i} \cap \overline{G_j} = \emptyset$ for every $i \neq j$, it follows straightforwardly that
\begin{equation}\label{lies}
(\overline{G} \cap H_\lambda)^\lambda \subset G \cap H_\lambda^\lambda \quad \text{with strict inclusion}.
\end{equation}
Furthermore, $\bar \lambda$ can be characterized as
\[
\bar \lambda= \inf \left\{ \lambda \in \R \left| \begin{array}{l}
\text{$(\overline{G} \cap H_\mu)^\mu \subset G \cap H_\mu^\mu$ with strict inclusion, and} \\
\text{$\langle \nu(x),e_N \rangle >0$ for every $x \in T_\mu \cap \partial \Omega$, for every $\mu >\lambda$}
\end{array} \right.\right\}.
\]
This simple observation permits to treat the case of multi-connected interior sets $G$ essentially as if they were consisting of only one domains.

Two preliminary results regarding the geometry of the reduced half-space are contained in the following statement.

\begin{lemma}\label{lem: geom remark}
For every $\lambda \ge \bar \lambda$, the following properties hold:
\begin{itemize}
\item[($i$)] $\overline{G} \cap H_\lambda$ is convex in the $e_N$-direction;
\item[($ii$)] the set $(\R^N \setminus \overline{G}) \cap H_\lambda$ is connected.
\end{itemize}
\end{lemma}

\begin{proof}
For property ($i$), we show that for any $\lambda \ge \bar \lambda$, if a point~$x=(x',x_N)$
belongs to~$\overline{G}\cap H_\lambda$, then also~$(x',t)\in \overline{G} \cap H_\lambda$
for every~$t\in [\lambda,x_N)$. If this is not true, then there exist $(x',x_N) \in G \cap H_\lambda$ and $(x',t) \not \in G$ with $t \in [\lambda,x_N)$. For $\lambda':= (x_N+t)/2 \ge \bar \lambda$, we have that $(x',t)= (x',x_N)^{\lambda'}$, but $(x',t) \not \in G$, which is in contradiction with the fact that, since $\lambda'>\bar \lambda$, the reflection of~$G$ with respect
to~$T_{\lambda'}$ does not exit~$\overline{G}$
itself. 

As far as property ($ii$) is concerned, given two points~$x^{(1)}$, $x^{(2)}\in  \Omega \cap H_\lambda$,
we fix a large~$M>0$
and consider the two vertical segments~$x^{(i)}+t e_N$,
with~$t\in[0, M-x^{(i)}_n]$. By point ($i$),
these segments lie in~$\Omega\cap H_\lambda$. Each segment connects~$x^{(i)}$
with~$y^{(i)}:= ((x^{(i)})', M)$.
Then, since~$G$ is bounded, if~$M$ is large we can connect~$
y^{(1)}$ and~$y^{(2)}$ with a horizontal segment~$ty^{(2)}+(1-t)y^{(1)}$
lying well outside~$G$. In this way, by considering the two vertical
segments and the horizontal one as a single polygonal, we have joined~$x^{(1)}$ and~$x^{(2)}$ by
a continuous path that lies in~$\Omega\cap H_\lambda$.
This shows that~$\Omega\cap H_\lambda$ is connected.
\end{proof}

We define $w_\lambda(x):= u(x^\lambda)-u(x)$. Notice that
\begin{equation}\label{ob1}
(-\Delta)^s w_\lambda + c_\lambda(x) w_\lambda=0 \qquad \text{in $\Sigma_\lambda$},
\end{equation}
where
\begin{equation}\label{def: c_lambda}
c_\lambda(x):= \begin{cases} -\displaystyle\frac{f(u_\lambda(x)) -f(u(x))}{u_\lambda(x)-u(x)} & \text{if $u_\lambda(x) \neq u(x)$} \\
0 & \text{if $u_\lambda(x) = u(x)$,} \end{cases}
\end{equation}
is in $L^\infty(\R^N)$ since $u$ is bounded and $f$ is locally Lipschitz continuous.

We aim at proving that the set
\begin{equation}\label{def: capital lambda}
\Lambda:=\left\{ \lambda > \bar \lambda: \text{$w_\mu \ge 0$ in $\Sigma_\mu$ for every $\mu \ge \lambda$} \right\}
\end{equation}
coincides with the interval $(\bar \lambda,+\infty)$, that $w_{\lambda}>0$ for every $\lambda \in \Lambda$, and that $w_{\bar \lambda} \equiv 0$ in $\Sigma_{\bar \lambda}$. From this we deduce that $u$ is symmetric with respect to $T_{\bar \lambda}$, and non-increasing in the $e_N$ direction in the half-space $H_\lambda$. Furthermore, we shall deduce that $\Omega$ (and hence also $G$) is convex in the $e_N$ direction, and symmetric with respect to $T_{\bar \lambda}$. As a product of the convexity and of the fact that $w_\lambda>0$ in $\Sigma_\lambda$ for $\lambda> \bar \lambda$, it is not difficult to deduce that $u$ is strictly decreasing in $x_N$ in $(\R^N \setminus \overline{G}) \cap H_{\bar \lambda}$. Repeating the same argument for all the directions $e \in \mathbb{S}^{N-1}$, we shall deduce the thesis.

Although the strategy of the proof is similar to
that of Theorem 1.1 in \cite{Reichel1}, its intermediate steps will differs 
substantially.

We write that the hyperplane $T_\lambda$ moves, and reaches a position $\mu$, if $w_\lambda \ge 0$ in $\Sigma_\lambda$ for every $\lambda >\mu$. With this terminology, the first step in the previous argument consists in showing that the movement of the hyperplane can start.

\begin{lemma}\label{lem: moving initial}
There exists $R>0$ sufficiently large such that $w_\lambda \ge 0$ in $\Sigma_\lambda$ for every $\lambda >R$.
\end{lemma}
\begin{proof}
We argue by contradiction, assuming that for a sequence $\lambda_k \to +\infty$ there exists $x_k \in \Sigma_{\lambda_k}$ such that $w_{\lambda_k}(x_k)<0$. Since $w_{\lambda_k}=0$ on $\pa \Sigma_{\lambda_k}=T_{\lambda_k}$, and $w_{\lambda_k} \to 0$ as $|x|$ tends to infinity, we can suppose that each $x_k$ is an interior minimum point of $w_{\lambda_k}$ in $\Sigma_{\lambda_k}$. 
Notice that
\begin{equation}\label{io}
c_{\lambda_k}(x_k) w_{\lambda_k}(x_k)\le0.\end{equation}
Indeed, on one side $w_{\lambda_k}(x_k) < 0$, on the other side since $|x_k| \to +\infty$ we have both $u_k(x_k) \to 0$ and $u_k(x_k^{\lambda_k}) \to 0$, and since $f$ is monotone non-increasing for small value of its argument, we deduce by \eqref{def: c_lambda} that $c_{\lambda_k}(x_k) \ge 0$. Now we show that
\begin{equation}\label{io2}
(-\Delta)^s w_{\lambda_k}(x_k)\le0.\end{equation}
For this, we consider the sets~$U:=\{ w_{\lambda_k} < w_{\lambda_k}(x_k)\}$
and~$V:= \{x_N<\lambda_k\}$. 
Notice that, by
the minimality property of~$x_k$ in~$\overline{H_{\lambda_k}}$,
we have that~$U\subset V$.
Therefore any integral in~$\R^N$ may be decomposed as the
sum of four integrals, namely the ones over~$U$, $V\setminus U$, $U^{\lambda_k}$
and~$H_{\lambda_k}\setminus U^{\lambda_k}$.
Using this and the fact that~$w_{\lambda_k}\geq w_{\lambda_k}(x_k)$
outside~$U$, we have that
\begin{equation}\label{io3}
\begin{split}
\int_{\R^N}\frac{w_{\lambda_k}(x_k)-w_{\lambda_k}(y)}{|x_k-y|^{N+2s}}\,dy
\le \int_{U} & \frac{w_{\lambda_k}(x_k)-w_{\lambda_k}(y)}{|x_k-y|^{N+2s}}\,dy \\
&+\int_{U^{\lambda_k}}\frac{w_{\lambda_k}(x_k)-w_{\lambda_k}(y)}{|x_k-y|^{N+2s}}\,dy.
\end{split}
\end{equation}
Also, if~$y\in U\subseteq \{x_N\le\lambda_k\}$ we have that~$|x_k-y|\ge
|x_k-y^{\lambda_k}|$, and therefore
$$ \int_{U}\frac{w_{\lambda_k}(x_k)-w_{\lambda_k}(y)}{|x_k-y|^{N+2s}}\,dy
\le \int_{U}\frac{w_{\lambda_k}(x_k)-w_{\lambda_k}(y)}{
|x_k-y^{\lambda_k}|^{N+2s}}\,dy,$$
since the numerator of the integrand is positive in~$U$.
By changing variable~$z:=y^{\lambda_k}$ in the latter integral,
we obtain
\begin{equation}\label{io4} 
\int_{U}\frac{w_{\lambda_k}(x_k)-w_{\lambda_k}(y)}{|x_k-y|^{N+2s}}\,dy
\le \int_{U^{\lambda_k}}\frac{w_{\lambda_k}(x_k)-w_{\lambda_k}(z^{\lambda_k})}{
|x_k-z|^{N+2s}}\,dz.\end{equation}
Observing that~$w_{\lambda_k}(z^{\lambda_k})=u(z)-u(z^{\lambda_k})=-w_{\lambda_k}(z)$
(and renaming the last variable of integration), we can write~\eqref{io4}
as
\begin{equation*}
\int_{U}\frac{w_{\lambda_k}(x_k)-w_{\lambda_k}(y)}{|x_k-y|^{N+2s}}\,dy
\le \int_{U^{\lambda_k}}\frac{w_{\lambda_k}(x_k)+w_{\lambda_k}(y)}{
|x_k-y|^{N+2s}}\,dy.\end{equation*}
By plugging this information into~\eqref{io3}, we conclude that
\begin{equation}\label{conclusion computation delta s}
(-\Delta)^s w_{\lambda_k}(x_k) 
%%\int_{U^{\lambda_k}}\frac{w_{\lambda_k}(x_k)+w_{\lambda_k}(y)}{
%|x_k-y|^{N+2s}}\,dy
%+
%\int_{U^{\lambda_k}}\frac{w_{\lambda_k}(x_k)-w_{\lambda_k}(y)}{|x_k-y|^{N+2s}}\,dy \\
\le 2\int_{U^{\lambda_k}}\frac{w_{\lambda_k}(x_k)}{
|x_k-y|^{N+2s}}\,dy\le0,
%\end{split}
\end{equation}
with strict inequality if $U \neq \emptyset$ since $w_{\lambda_k}(x_k)<0$. This proves~\eqref{io2}.

Now we claim that
\begin{equation}\label{ob2}
(-\Delta)^s w_{\lambda_k}(x_k)=0=c_{\lambda_k}(x_k) w_{\lambda_k}(x_k).
\end{equation}
Indeed, we already know that
both the quantities above are non-positive, due to~\eqref{io}
and~\eqref{io2}. If at least one
of them were strictly negative, 
their sum would be strictly negative too, and this is in contradiction
with~\eqref{ob1}. Having established~\eqref{ob2}, we use it to observe that~$w_{\lambda_k}$
must be constant. Indeed, we firstly notice that $U = \emptyset$ (otherwise, as already observed, in~\eqref{conclusion computation delta s} we would have a strict inequality).
This means that~$w_{\lambda_k} \ge w_{\lambda_k}(x_k)$ in the whole $\R^N$, and if the set~$\{w_{\lambda_k}>w_{\lambda_k}(x_k)\}$ would have positive
measure, this would imply that~$(-\Delta)^s w_{\lambda_k}(x_k)<0$.

Thus we conclude that~$w_{\lambda_k} \equiv w_{\lambda_k}(x_k) <0$, in contradiction with its anti-symmetry with respect to $T_{\lambda_k}$.
\end{proof}

Thanks to the above statement, the value $\mu:= \inf \Lambda$ is a real number. We aim at proving that the hyperplane $T_\lambda$ reaches the position $\bar \lambda$, i.e. $\mu = \bar \lambda$. In this perspective, a crucial intermediate result is the following.

\begin{lemma}\label{lem: condizione per simmetria}
Let $\lambda \ge \bar \lambda$. If $w_\lambda(x) = 0$ for some $x \in \Sigma_{\lambda}$, then $G$ is symmetric with respect to $T_\lambda$. In particular, if $w_\lambda \ge 0$ in $\Sigma_\lambda$ and $\lambda>\bar \lambda$, then $w_\lambda>0$ in $\Sigma_\lambda$.
\end{lemma}

\begin{proof}
By the strong maximum principle, Proposition \ref{STRONG}, we have that if $w_\lambda(x) = 0$ for a point $x \in \Sigma_\lambda$, then $w_\lambda \equiv 0$ in $H_\lambda$. Let us assume by contradiction that $G$ is not symmetric with respect to $T_\lambda$. At first, it is easy to check that 
\[
G \cap H_\lambda^\lambda = (G \cap H_\lambda)^\lambda \quad \Longrightarrow \quad G= G^\lambda.
\]
Hence, having assumed that $G$ is not symmetric we have $G \cap H_\lambda^\lambda \neq (G \cap H_\lambda)^\lambda$; since for $\lambda > \bar \lambda$ the inclusion \eqref{lies} holds, this implies that $\overline{G} \cap H_\lambda^\lambda \supset (\overline{G} \cap H_\lambda)^\lambda$ with strict inclusion. Let
\[
E:=  (\overline{G} \cap H_\lambda^\lambda ) \setminus (\overline{G} \cap H_\lambda)^\lambda \neq \emptyset.
\]
\begin{figure}[ht]
\includegraphics[height=5cm]{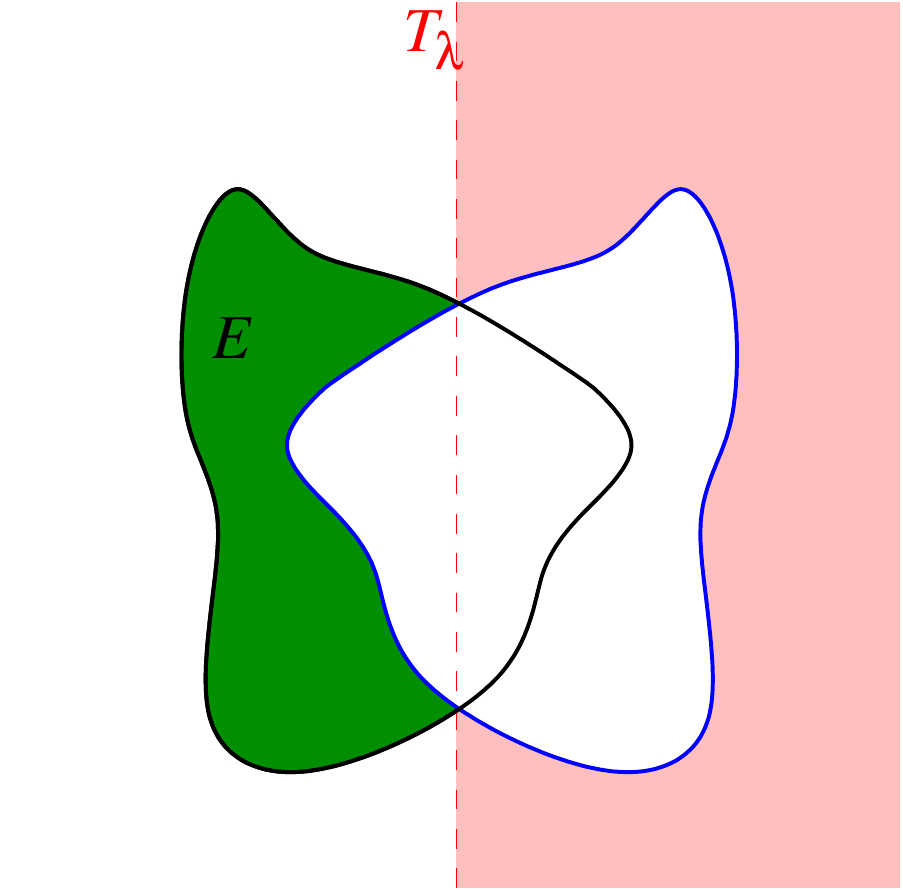} \hspace{1 cm} \includegraphics[height=5cm]{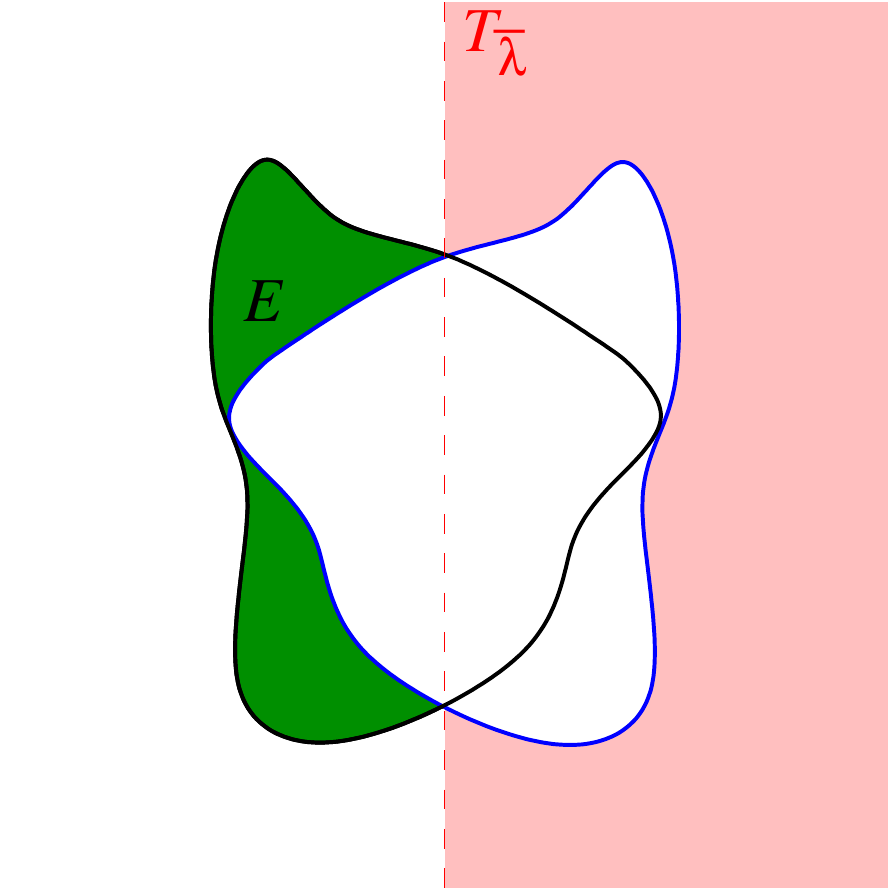}
\caption{On the left, it is represented the situation for $\lambda> \bar \lambda$; on the right, the one for $\lambda=\bar \lambda$ with $G$ not symmetric with respect to $T_{\bar \lambda}$.}
\end{figure}

\noindent For every $x \in E^\lambda \subset H_\lambda$, we have 
\begin{align*}
x^\lambda \in E \subset \overline{G} \quad \Longrightarrow \quad u(x^\lambda)=a, \\
x \in \R^N \setminus \overline{G} \quad \Longrightarrow \quad u(x)<a,
\end{align*}
and hence $w_\lambda(x) >0$, a contradiction. 

So far we showed that if $\lambda \ge \bar \lambda$ and $w_\lambda$ vanishes somewhere in $\Sigma_\lambda$, then $G$ is symmetric with respect to $T_\lambda$. But, by definition of $\bar \lambda$, this cannot be the case for $\lambda>\bar \lambda$, and hence for any such $\lambda$ if we can prove that $w_\lambda \ge 0$ in $\Sigma_\lambda$, we immediately deduce that $w_\lambda>0$ there. 
\end{proof}

\begin{remark}
In the proof we used in a crucial way both the nonlocal boundary conditions, and the nonlocal strong maximum principle. In particular, we point out that if $w_\lambda(x) = 0$ at one point $x \in \Sigma_\lambda$, then $w_\lambda \equiv 0$ in the whole half-space $H_\lambda$.
\end{remark}

In the following we shall make use the fact that from the sign of the $s$-derivative of a function $u$ in a given direction we can infer information about the monotonicity of $u$ itself. 

\begin{lemma}\label{lem: monotonia}
 Let~$\Omega$ be an open subset of~$\R^N$
with~${\mathcal{C}}^2$ boundary, $x_0 \in \pa \Omega$, and let $\nu(x_0)$ denote the inner unit normal vector to $\pa \Omega$ in $x_0$.
Let $w \in \mathcal{C}^s(\R^N) \cap \mathcal{C}^1(\Omega)$
be such that~$w=0$ in~$\R^N\setminus\Omega$.
Let~$d(x)$ to be the distance of~$x\in\overline\Omega$
to the boundary of~$\Omega$ and
$$ W(x):=\frac{w(x)}{\big( d(x)\big)^s }.$$
Assume that
\begin{equation}\label{W buona}
W\in {\mathcal{C}}^{0,1}(\overline\Omega)\quad \text{and} \quad W(x_0) <0.
\end{equation}
If $\eta \in \mathbb{S}^{N-1}$ is such that $\eta \cdot \nu(x_0)  >0$, then $w$ is monotone decreasing in the direction $\eta$ in a neighborhood of $x_0$. 
\end{lemma}

\begin{remark}
Clearly, the condition $w=0$ in $\R^N \setminus \Omega$ can be replaced by $w=const.$ in $\R^N \setminus \Omega$, without changing the thesis.
\end{remark}

\begin{proof}
%For every~$x\in\Omega$ and~$t>0$ (small enough),
%there exists~$\tau\in[0,t]$ such that
Let $x \in \Omega$ and $t>0$. We denote by $P_G$ the projection at minimal distance from $\Omega$ to $\pa \Omega$, which is well defined and continuous close to $x_0$ since $\Omega$ is of class $\mathcal{C}^2$. The gradient of the distance function $d$ can be expressed as
\[
\nabla d(x+t\eta) =  \frac{x+t\eta - P_G(x+t\eta)}{|x+t\eta - P_G(x+t\eta)|} = \nu(P_G(x+t\eta)),
\]
see \cite[Theorem 4.8]{Fed}. Using the continuity of $\nu(\cdot)$ on $\pa \Omega$, we have then the following expansion for the distance $d(x+t\eta)$:
\begin{equation}\label{2111}
\begin{split}
d(x+t \eta)-d(x) &=
t \nabla d(x+\tau  t \eta)\cdot \eta =
t \, \nu(P_G(x+t \tau \eta)) \cdot \eta  \\
& = t \, \nu(x_0) \cdot \eta + o(t) 
 \ge \frac{\nu(x_0) \cdot \eta}{2} t > 0,
 \end{split}
\end{equation}
where $\tau \in [0,1]$ and $|x-x_0| + |t|$ is small enough. Furthermore, in view of~\eqref{W buona}, we know that there exist $C, c_*>0$ such that
\begin{equation}\label{2112}
\big|W(x+t\eta)-W(x)\big| \le C t \quad \text{and} \quad W(x+t\eta) \le -c_*
\end{equation}
if~$|x-x_0|+|t|$ is small enough. Therefore, there exists $\delta>0$ such that if $|x-x_0|<\delta$ and $0<t<\delta/2$, both \eqref{2111} and \eqref{2112} holds true. 

We fix now $x \in \Omega$ with $|x-x_0|<\delta/2$, and for any $t$ small we observe that, by \eqref{2111}, for some $\tau \in [0,1]$
\begin{equation}\label{RE4193}
\begin{split}
 \big( d(x+t\eta)\big)^s-&\big( d(x)\big)^s \ge 
\left( d(x)+t\, \frac{
\nu(x_0)\cdot \eta}{2} \right)^s- \big( d(x)\big)^s\\
&= s \frac{ \frac12 (\nu(x_0) \cdot \eta)t }{\left( d(x) + \frac12(\nu(x_0) \cdot \eta) \tau t \right)^{1-s}}   \ge s \frac{ \frac12 (\nu(x_0) \cdot \eta)t }{\left( d(x) + \frac12(\nu(x_0) \cdot \eta) t \right)^{1-s}}.
\end{split}
\end{equation}
%
%\big( d(x) + (\nu \cdot \eta) \tau t + \tau o(t) \big)^{s-1} \big( (\nu \cdot \eta)t + o(t) \big) 
%& \qquad =
%\big( d(x)\big)^s \,\left[\frac{ s\,t\, (
%\nu\cdot \eta)}{2\,d(x)} +O\left( \frac{t^2}{\big(d(x)\big)^2}\right)
%\right]
Now \eqref{2112} and \eqref{RE4193} implies that, for any~$x\in\Omega$ with $|x-x_0|<\delta/2$ and for~$t\to0^+$,
\begin{align*}
w(x+t\eta) &-w(x) = W(x+t\eta)\big( d(x+t\eta)\big)^s-W(x)\big( d(x)\big)^s\\ 
& = W(x+t\eta)\Big[
\big( d(x+t\eta)\big)^s-\big( d(x)\big)^s\Big]
+\big( d(x)\big)^s\Big[W(x+t\eta)-W(x)\Big]\\
&\le - s \, c_\star \frac{ \frac12 (\nu(x_0) \cdot \eta)t }{\left( d(x) + \frac12(\nu(x_0) \cdot \eta) t \right)^{1-s}}  + C t \big(d(x)\big)^s.
\end{align*}
So, dividing by~$t>0$ and taking the limit, we find that for every $x \in \Omega$ with $|x_0-x_0|<\delta$
\[
\partial_\eta w(x) =
\lim_{t\to0^+}\frac{ w(x+t\eta)-w(x) }{t}\le - s \, c_\star \frac{\frac12(\nu(x_0) \cdot \eta)}{\big( d(x) \big)^{1-s}} + C \big(d(x)\big)^s.
\]
If necessary replacing $\delta$ with a smaller quantity, we see that the right hand side is strictly negative, and this gives the desired claim.
\end{proof}

We are finally ready to prove that the hyperplane $T_\lambda$ reaches the critical position $\bar \lambda$.

\begin{lemma}\label{lem: moving continuation}
There holds $\inf \Lambda=\bar \lambda$.
\end{lemma}

\begin{proof}
By contradiction, we suppose that $\mu=\inf \Lambda> \bar \lambda$. By continuity, we have $w_\mu \ge 0$ in $\Sigma_\mu$, and hence Lemma \ref{lem: condizione per simmetria} yields $w_\lambda>0$ in $\Sigma_\lambda$. Since $\mu=\inf\Lambda$, there exist sequences $\bar \lambda < \lambda_k < \mu$, $\lambda_k \to \mu$, and $x_k \in \Sigma_{\lambda_k}$, such that $w_{\lambda_k}(x_k) <0$. Since $w_{\lambda_k} \ge 0$ on $\partial \Sigma_{\lambda_k}$ and $w_{\lambda_k} \to 0$ as $|x| \to +\infty$, it is not restrictive to assume that $x_k$ are interior minimum point of $w_{\lambda_k}$ in $\Sigma_{\lambda_k}$. If $|x_k| \to +\infty$, we obtain a contradiction as in Lemma \ref{lem: moving initial}. If $\{x_k\}$ is bounded, there exists $\bar x \in \Sigma_{\mu}$ such that, up to a subsequence, $x_k \to  \bar x$. Concerning the pre-compactness of the sequence $\{w_{\lambda_k}\}$, we recall that by definition
\[
w_{\lambda_k}(x) = u(x',2\lambda_k-x_N)-u(x',x_N).
\]
The function $u$ is of class $\mathcal{C}^s(\R^N)$, and hence for every compact $K \Subset \R^N$ there exists $C>0$ such that 
\[
\|w_{\lambda_k}\|_{\mathcal{C}^s\left(K\right)} \le C.
\]
Therefore, the sequence $w_{\lambda_k}$ is convergent to $w_{\mu}$ in $\mathcal{C}^{s'}_{\loc}(\R^N)$ for any $0<s'<s$. The uniform convergence entails $w_\mu(\bar x) = 0$, and by Lemma \ref{lem: condizione per simmetria} this implies that $\bar x \in \partial \Sigma_{\mu}$ is a boundary point. To continue the proof, we have to distinguish among three different possibilities, and in each of them we have to find a contradiction.

\emph{Case 1) $\bar x$ lies on the regular part of $\partial \Sigma_{\mu} \cap T_{\mu}$.} We note that 
\[
\bar x \in \interior\left(\overline{\Sigma_{\lambda_k} \cup \Sigma_{\lambda_k}^{\lambda_k}}\right)
\]
for every $k$ sufficiently large, where $\interior$ denotes the interior of a set. Since, by interior regularity (see Subsection \ref{sub: regularity}), we know that $u \in \mathcal{C}^{1,\sigma}(\R^N \setminus \overline{G})$ for some $\sigma \in (0,1)$, by definition $\{w_k\}$ is uniformly bounded in $\mathcal{C}^{1,\sigma}(\overline{B_\rho(\bar x)})$. In particular, since by interior minimality (and interior regularity) we have $\nabla w_{\lambda_k}(x_k) = 0$ and $w_{\lambda_k} \to w_{\mu}$ in $\mathcal{C}^1(\overline{B_\rho(\bar x)})$, we deduce that $\nabla w_{\mu}(\bar x)=0$. This is in contradiction with Proposition \ref{prop: new hopf}, where we have showed that 
\[
-\liminf_{t \to 0^+} \frac{w_{\mu}(\bar x', \bar x_N +t)}{t} < 0.
\]

\emph{Case 2) $\bar x \in \partial (G^\mu \cap H_{\mu}) \setminus T_\mu$.} Since $\mu > \bar \lambda$, we know that $\pa (G^\mu \cap H_{\mu}) \cap \overline{G} = \emptyset$ (otherwise we have to be in a critical position of internal tangency). Having assumed that $0 \le u <a$ in $\Omega$, we deduce that $u(\bar x) <a$, and hence 
\[
w_{\mu}(\bar x) = a- u(\bar x) >0,
\]
in contradiction with the fact that by convergence $w_{\mu}(\bar x) = 0$.
 
\emph{Case 3) $\bar x \in \pa G^{\mu} \cap T_\mu$.} We observe that $\bar x \in \partial G$, and since $\mu>\bar \lambda$, the outer (with respect to $G$) unit normal vector $\nu(\bar x)$ is such that $\langle \nu(\bar x), e_N \rangle > 0$. By assumption \eqref{W buona intro}, the function $W := u/d^s$ is Lipschitz continuous in $\overline{\Omega}$, and, by the Hopf lemma for the fractional Laplacian, we have that $u(x)-a \le -C (d(x))^s$ close to $\pa G$. Therefore, by Lemma \ref{lem: monotonia},
we see that $\pa_{e_N} u < 0$ in $B_{\rho}(\bar x) \cap \R^N \setminus G$ for some positive $\rho$. Let $y_k$ denote the reflection of $x_k$ with respect to $T_{\lambda_k}$. Since both $x_k,y_k \to \bar x$, at least for $k$ sufficiently large the whole segment connecting $x_k$ with $y_k$ is contained in $B_{\rho}(\bar x) \cap \overline{\R^N \setminus \overline{G}}$. Recalling that $\pa_{e_N} u < 0$, this implies that $u$ is monotone decreasing along the segment $[y_k,x_k]$, that is
\[
w_{\lambda_k}(x_k) = u(y_k)-u(x_k) > 0,
\]
in contradiction with the fact that $w_{\lambda_k}(x_k) <0$.
\end{proof}

\begin{proof}[Conclusion of the proof of Theorem \ref{thm: main 1}]
For every $\lambda \ge \bar \lambda$, we have $w_\lambda \ge 0$ in $\Sigma_\lambda$. If $\lambda>\bar \lambda$, then the strict inequality holds, and it remains to show that $w_{\bar \lambda} \equiv 0$ in $\Sigma_{\bar \lambda}$. To this aim, we argue by contradiction assuming that $w_{\bar \lambda} >0 $ in $\Sigma_{\bar \lambda}$, and we distinguish two cases. The following argument is adapted by \cite{FallJarohs}.

\emph{Case 1) $\bar \lambda$ is a critical value of internal tangency.} There exists $x_0 \in \partial G \cap \partial( (G \cap H_\lambda)^\lambda) \setminus T_\lambda$. Clearly we have $w_{\bar \lambda}(x_0)=0$, so that $x_0^\lambda \in \partial G \cap H_\lambda$, and by Proposition \ref{STRONG} we deduce that $(\pa_{\nu})_s w_{\bar \lambda}(x_0) <0$, where $\nu$ denotes the outer unit normal vector to $\pa G$ in $x_0$. 

We observe that it cannot be $x_0 \in \partial G_i \cap H_{\bar \lambda}^{\bar \lambda}$ and $x_0 \in (\pa G_j \cap H_{\bar \lambda})^{\bar \lambda}$ with $i \neq j$. This follows from the definition of $\bar \lambda$ and the fact that $\overline{G_i} \cap \overline{G_j} = \emptyset$, see \cite[Lemma 2.1]{Sirakov} for a detailed proof. Hence, having assumed that $(\pa_\nu)_s u = \alpha_i$ on $\pa G_i$, and observing that by internal tangency $\nu(x_0)= - \nu(x_0^\lambda)$, we find also 
\[
(\pa_{\nu(x_0)})_s w_{\bar \lambda}(x_0) =  (\pa_{\nu(x_0^\lambda)})_s u(x_0^{\bar \lambda}) - (\pa_{\nu(x_0)})_s u(x_0) = 0,
\] 
which is a contradiction.

\emph{Case 2) $\bar \lambda$ is a critical value where the orthogonality condition is satisfied.} Let $x_0 \in \partial G_i$, and let $\delta=\delta_{G_i}$ denote the distance function from the boundary of $G_i$. Up to rigid motions, it is possible to suppose that $x_0=0$, $\nu(x_0)$ is a vector of the orthonormal basis, say $\nu(x_0)=e_1$, and $\nabla^2 \delta(0)$ is a diagonal matrix; to ensure that the function $\delta$ is twice differentiable, we used the $\mathcal{C}^2$ regularity of $G$. Let $\eta:=(1,0,\dots,0,1)$. Adapting step by step the proof of Lemma 4.3 in \cite{FallJarohs}, it is possible to deduce that $w_{\bar \lambda}(t \eta) = o(t^{1+s})$ as $t \to 0^+$. In this step we need to recall that $u/\delta^s \in \mathcal{C}^{0,\gamma}(\R^N \setminus G)$, see Subsection \ref{sub: regularity}. On the contrary, thanks to the nonlocal version of the Serrin's corner lemma (Lemma 4.4 in \cite{FallJarohs}; the result is stated therein in a bounded domain, but the reader can check that this is not use in the proof) we also infer that $w_{\bar \lambda}(t \eta)  \ge C t^{1+s}$ for $t$ positive and small, for some constant $C>0$. This gives a contradiction. 

We proved that $w_{\bar \lambda} \equiv 0$. By Lemma \ref{lem: condizione per simmetria}, this implies that $G$, and hence $\R^N \setminus \overline{G}$, are symmetric with respect to the hyperplane $T_{\bar \lambda}$. In principle both $\R^N \setminus \overline{G}$ and $G$ could have several connected components. But, as proved in Lemma \ref{lem: geom remark}, $(\R^N \setminus \overline{G}) \cap H_{\bar \lambda}$ is connected, which implies by symmetry that $(\R^N \setminus \overline{G}) \cap H_{\bar \lambda}^{\bar \lambda}$ is in turn connected. Therefore, if $\R^N \setminus \overline{G}$ is not connected, necessarily $G$ contains a neighborhood of the hyperplane $T_{\bar \lambda}$, which is not possible since $G$ is bounded.

As far as the connectedness of $G$ is concerned, we firstly observe that by property ($i$) of Lemma \ref{lem: geom remark} and by symmetry, $G$ is convex in the $e_N$-direction. Let us assume by contradiction that there exists at least two connected components $G_1$ and $G_2$ of $G$. It is not possible that $G_1$ and $G_2$ meet at boundary points, since we assumed that $G$ is of class $\mathcal{C}^2$. 
%Furthermore, it cannot be that $G_1$ surround $G_2$, because otherwise the convexity would be contradicted. 
Since $G$ is convex, there exists a hyperplane $T'$ not parallel to $T_{\bar \lambda}$ separating $G_1$ and $G_2$. Let $e$ be an orthogonal direction to $T'$. Defining
\begin{align*}
T_\lambda'&:= \{x \in \R^N: \langle x,e\rangle > \lambda\} \\
d'&:= \inf\left\{ \lambda \in \R: T_\mu' \cap \overline{G} = \emptyset \text{ for every $\mu >\lambda$} \right\},
\end{align*}
without loss of generality we can suppose that $\overline{G_1} \cap T_{d'}'\neq \emptyset$, while $\overline{G_2} \cap T_{d'}' =\emptyset$. In the same way as we defined $\bar \lambda$ for the direction $e_N$, we can now define $\bar \lambda'$ for $e$, and prove that $G$ is symmetric with respect to $T'_{\bar \lambda'}$. But this gives clearly a contradiction, since by definition $G_2 \cap \{\langle x,e \rangle \ge \bar \lambda'\} = \emptyset$, while $G_2 \cap \{\langle x,e \rangle < \bar \lambda'\} = G_2$. 
\end{proof}

\subsection{Proof of Theorem \ref{thm: main 1 prime}}

The proof of Theorem \ref{thm: main 1 prime} follows a different sketch, being based upon the following known result (we refer to the appendix in \cite{FallJarohs} for a detailed proof).

\begin{proposition}\label{prop: criterion}
Let $u : \R^N \to \R$ be continuous and such that $u$ has a limit when $|x| \to +\infty$. Then the following statements are equivalent:
\begin{itemize}
\item[($i$)] $u$ is radially symmetric and radially non-increasing with respect to a point of $\R^N$;
\item[($ii$)] for every half-space $H$ of $\R^N$ we have that either $u(x) \ge u(R_H(x))$ in $H$, or $u(x) \le u(R_H(x))$ in $H$, where $R_H$ denotes the reflection with respect to the boundary $\partial H$.
\end{itemize}
\end{proposition}

Hence, to prove the radial symmetry of $u$ we aim at showing that condition ($ii$) in Proposition~\ref{prop: criterion}
is satisfied. 

\medskip

Let us consider at first all the half-space $H$ such that $\partial H$ is orthogonal to the $e_N$ direction. Using the notation introduced at the beginning of this section, and recalling the definition \eqref{def: c_lambda} of $c_\lambda$, we see that $c_\lambda \ge 0$ in $\R^N$ for every $\lambda$. Thus, it is not difficult to adapt the proof of Lemma \ref{lem: moving initial}, using the fact that for $\lambda \ge \bar \lambda$ we have $w_\lambda \ge 0$ in $H_\lambda \setminus \Sigma_\lambda$, to deduce that 
\begin{equation}\label{GE0}
\text{for every $\lambda \ge \bar \lambda$, it results that $w_\lambda \ge 0$ in $\Sigma_{\lambda}$.}
\end{equation}

On the contrary, we point out that now we cannot immediately conclude that $w_\lambda>0$ in $\Sigma_\lambda$ for $\lambda>\bar \lambda$, since in the proof of Lemma \ref{lem: condizione per simmetria} we have used both the assumptions $\alpha<0$ and $u<a$ in $\R^N \setminus \overline{G}$. Nevertheless, we can prove that 
\begin{equation}\label{SIM}
\text{$w_{\bar \lambda} \equiv 0$ in $H_{\bar \lambda}$},
\end{equation}
arguing exactly as in the conclusion of the proof of Theorem \ref{thm: main 1}.

\medskip

This line of reasoning can be used for all the direction $e \in \mathbb{S}^{N-1}$. To be more explicit, we introduce the following notation: for a direction $e$ and a real number $\lambda$, we set
\begin{align*}
T_{e,\lambda} & :=\{\langle x,e \rangle= \lambda\} \\
H_{e,\lambda} & := \{ \langle x,e \rangle > \lambda \} \\
R_{e,\lambda} & := \text{reflection with respect to the hyperplane $T_{e,\lambda}$}\\
\Sigma_{e,\lambda} & := (\R^N \setminus \overline{R_{e,\lambda}(G)}) \cap H_{e,\lambda}  \\
\bar \lambda(e) &:= \text{critical position for the direction $e$} \\
x^{e,\lambda} &:= x + (2\lambda- \langle x,e \rangle) e = R_{e,\lambda}(x) \\
w_{e,\lambda}(x)& := u(x^{\lambda,e})- u(x). \\
\end{align*}
As in \eqref{GE0} and \eqref{SIM}, for every $e \in \mathbb{S}^{N-1}$ there hold
\[
\text{for every $\lambda \ge \bar \lambda(e)$, it results that $w_{e,\lambda} \ge 0$ in $\Sigma_{e,\lambda}$,}
\]

and 
\[
\text{$w_{e,\bar \lambda(e)} \equiv 0$ in $H_{e,\bar \lambda(e)}$}.
\]
We show that this implies that condition ($ii$) in Proposition \ref{prop: criterion}. The following lemma has been implicitly used in the proof of Theorem 5.1 in \cite{FallJarohs}; here we prefer to include a detailed proof for the sake of completeness. 

\begin{lemma}\label{090}
Let $e \in \mathbb{S}^{N-1}$, and let us assume that for $\lambda > \bar \lambda(e)$ it results $w_{e,\lambda} \ge 0$ in $H_{e,\lambda}$, and $w_{e,\bar \lambda(e)} \equiv 0$ in $H_{e,\bar \lambda(e)}$. Then 
\[
\text{either $w_{\mu} \ge 0$ in $H_{e,\mu}$, or $w_\mu \le 0$ in $H_{e,\mu}$, for every $\mu \in \R$}.
\]
\end{lemma}

Notice that in principle $\bar \lambda(e) \neq -\bar \lambda(-e)$, and hence the result is not immediate.

\begin{proof}
Only to fix our minds, we consider $e=e_N$, and for the sake of simplicity we omit the dependence on $e_N$ in the notation previously introduced. For $\lambda \ge \bar \lambda$ there is nothing to prove. Let $\lambda < \bar \lambda$. We fix $\mu= 2\bar \lambda - \lambda$, so that $\bar \lambda$ is the medium point between $\lambda$ and $\mu$. In this way we have
\begin{align*}
w_\lambda(x',x_N) &=  u(x',2\lambda-x_N)-u(x',x_N) = u\left(x',2(2\bar \lambda-\mu)-x_N\right) - u(x',x_N) \\
 &=  u\left(x',2 \bar \lambda-( 2\mu + x_N - 2\bar \lambda) \right)  -u(x',x_N) \\
 &= u(x',2\mu+x_N-2 \bar \lambda) - u(x',x_N) \\
& = u\left(x',2\mu-(2\bar \lambda-x_N) \right) -u(x',2\bar \lambda -x_N) \\
& = w_\mu(x',2\bar \lambda-x_N), 
\end{align*}
where we used the fact that $u(x^{\bar \lambda}) = u(x)$ for every $x \in \R^N$. Now it is sufficient to observe that if $x \in H_\lambda$, then 
\[
2\bar \lambda -x_N < 2 \bar \lambda-\lambda =\mu,
\]
that is, $(x',2\bar \lambda-x_N) \in \R^N \setminus \overline{H_\mu}$. Therefore, using the fact that $w_\mu \ge 0$ in $H_\mu$ and is anti-symmetric, we conclude that $w_\lambda \le 0$ in $H_\lambda$ for every $\lambda<\bar \lambda$.
\end{proof}
The result in Lemma~\ref{090} means that for every $e \in \mathbb{S}^{N-1}$ and $\lambda \in \R$ we have that either $u(x) \ge u(x^{e,\lambda})$ in $H_{e,\lambda}$, or else $u(x) \le u(x^{e,\lambda})$ in $H_{e,\lambda}$. Since the $H_{e,\lambda}$ are all the possible half-spaces of $\R^N$, by Proposition \ref{prop: criterion} we infer that $u$ is radially symmetric and radially non-increasing with respect to some point of $\R^N$. This still does not prove that $G$ is radially symmetric, but it is sufficient to ensure that $\{u <a\}$ is the complement of a ball $B$ or a certain radius $\rho$. Up to a translation, it is not restrictive to assume that the centre of $B$ is the origin. To complete the proof of Theorem \ref{thm: main 1 prime}, we have to show that $B=G$. 

\medskip

Before, we point out that since $u$ is radial and non-constant, if $w_{e,\lambda} \equiv 0$ in $H_{e,\lambda}$, then necessarily $\lambda=\bar \lambda(e)= 0$. Indeed, if this were not true, then there exists $x \in R_{e,\lambda}(\overline{B}) \setminus \overline{B} \neq \emptyset$, and for any such $x$ we have 
\[
0 = w_{e,\lambda }(x) = u( x^{e,\lambda})-u(x) = a- u(x) >0,
\]
a contradiction. Therefore, the strong maximum principle together with \eqref{GE0} imply that $w_{e,\lambda}>0$ in $\Sigma_{e,\lambda}$ for every $e \in \mathbb{S}^{N-1}$ and $\lambda>0$, which in particular proves the radial strict monotonicity of $u$ outside $\overline{B}$.

\medskip

Now we show that $B=G$. We argue by contradiction, noting that there are two possibilities: either $\overline{B \setminus \overline{G}} \cap \partial \{u<a\} \neq \emptyset$, or the intersection is empty, which means that
$G$ is an annular region 
surrounding $(\R^N \setminus \overline{G}) \setminus \{u<a\}$.

\begin{figure}[ht]
\includegraphics[height=4.5cm]{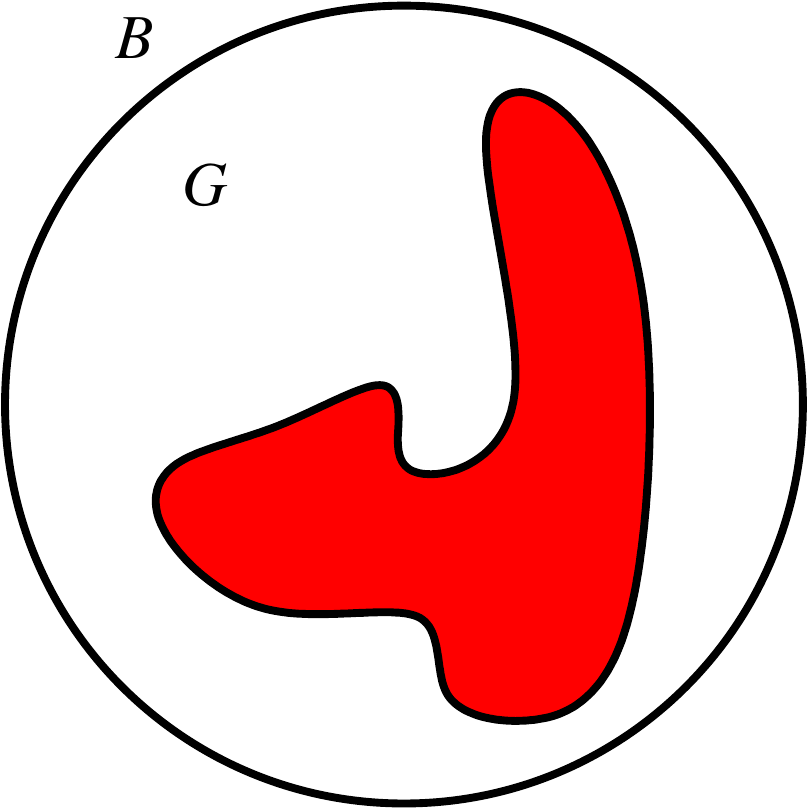}  \hspace{1 cm} \includegraphics[height=5cm]{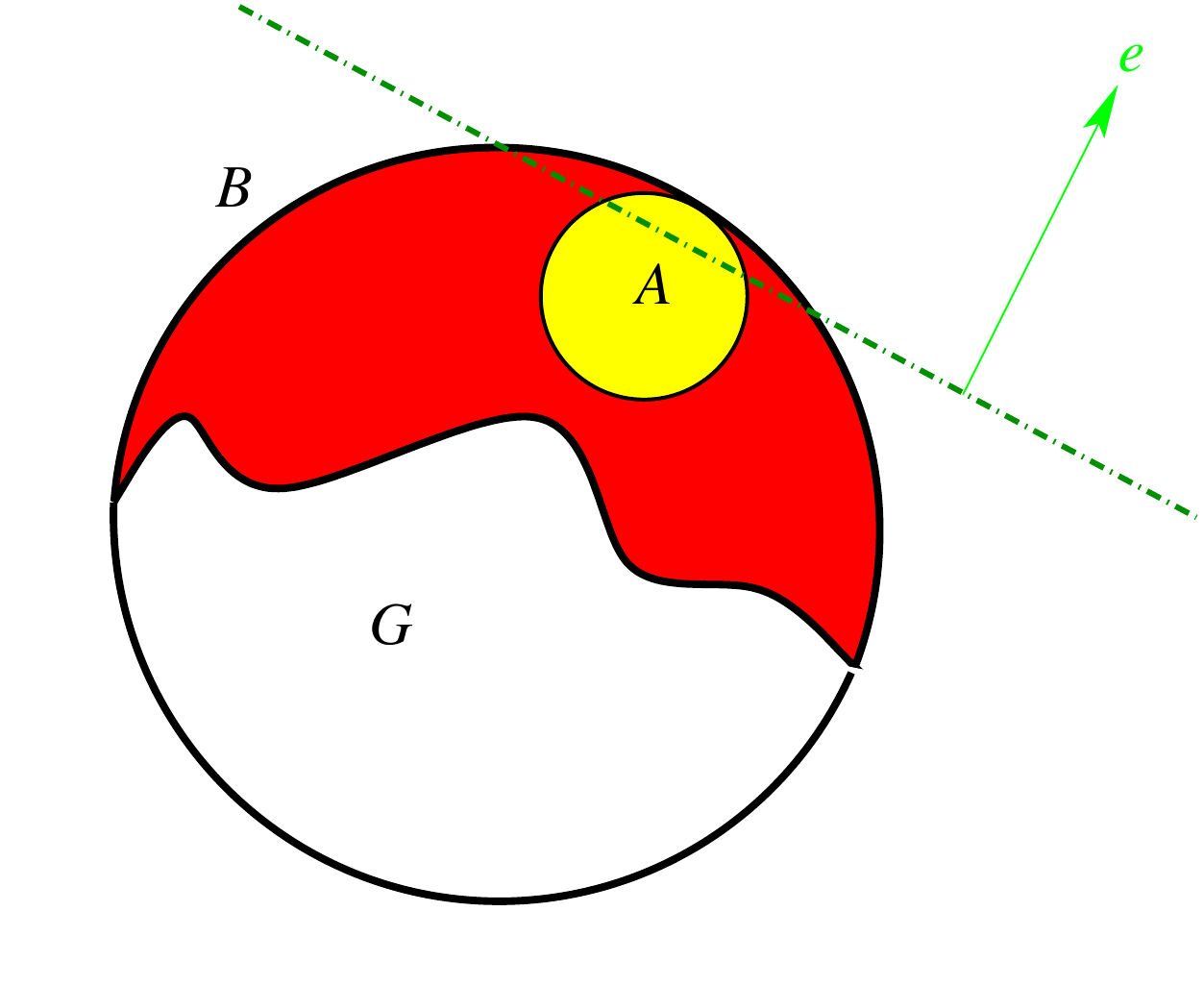} 
\caption{On the left, the case $\overline{B \setminus \overline{G}} \cap \partial \{u<a\} \neq \emptyset$; on the right, $\overline{B \setminus \overline{G}} \cap \partial \{u<a\} \neq \emptyset$.}
\end{figure}

%\begin{figure}[h]
%\includegraphics[height=3.5cm]{FIG4SV.pdf}
%\caption{Proof of Theorem \ref{thm: main 1 prime}:
%the case in which $G$ is an annular region
%surrounding $(\R^N \setminus \overline{G}) \setminus \{u<a\}$.}
%\end{figure}

If the latter alternative takes place, then there exists $e \in \mathbb{S}^{N-1}$ and $\lambda>0$ such that $\overline{G} \cap H_{e,\lambda}$ is not convex. This is in contradiction with Lemma \ref{lem: geom remark}.

%\begin{figure}[h]
%\includegraphics[height=3.5cm]{FIG5SV.pdf}
%\caption{Proof of Theorem \ref{thm: main 1 prime}:
%the case in which 
%$\overline{B \setminus \overline{G}} \cap \partial \{u<a\} \neq \emptyset$.}
%\end{figure}

It remains to show that also $
\overline{B \setminus \overline{G}} \cap \partial \{u<a\} \neq \emptyset$ cannot occur. To this aim, we observe that in such a situation there exists a direction $e \in \mathbb{S}^{N-1}$, a small interval $[\lambda_1,\lambda_2]$ with $\lambda_1>0$, and a small ball $A \subset B \setminus \overline{G}$, such that $\overline{A} \cap T_{e,\lambda} \neq \emptyset$ while $\overline{G} \cap T_{e,\lambda} \neq \emptyset$ for every $\lambda \in [\lambda_1,\lambda_2]$. Let $(\lambda_1+\lambda_2)/2 < \lambda < \lambda_2$. Then the reflection $R_{e,\lambda}(A \cap H_{e,\lambda}) \subset A$, and hence therein  we have $w_{e,\lambda}(x) = 0$. But on the other hand, since $A \cap H_{e,\lambda} \subset \Sigma_\lambda$, we have also that by the strong maximum principle $w_{e,\lambda} >0$ in $\Sigma_{e,\lambda}$, a contradiction (recall that the only $\mu$ such that $w_{e,\mu} \equiv 0$ in $H_{e,\mu}$ is $\mu=0$).

In this way, we have shown that $\overline{G} = \{u=a\}$ is a ball, and the function $u$ is radial with respect to the centre of $G$ and radially decreasing in $\R^N \setminus \overline{G}$, which is the desired result.

\subsection{Proof of Corollary \ref{corol: main 1}}

To show that $0 \le u \le a$ and $f(a) \le 0$ imply $u <a$ in $\R^N$, we use a comparison argument, as in the local case (see \cite[Corollary 1]{Reichel1}). Let us set
\begin{equation}\label{def c in corol}
c(x) := \begin{cases} -\displaystyle\frac{f(u(x))-f(a)}{u(x)-a} & \text{if $u(x) <a$} \\
0 & \text{if $u(x) = a$}.
\end{cases}
\end{equation}
Then, recalling that $u \le a$, we have
\begin{align*}
(-\Delta)^s (u-a) + c^+(x) (u-a) &\le   (-\Delta)^s (u-a) + c(x)(u-a) \\
& = (-\Delta)^s u - (  f(u)-f(a) )  \le 0 
\end{align*}
in $\R^N \setminus \overline{G}$. We claim that this implies $u<a$ in $\R^N \setminus \overline{G}$. Indeed, if this were not true that there would exists a point $\bar x \in \R^N \setminus \overline{G}$ for $u$ with $u(\bar x) = a$. But then $(-\Delta)^s u(\bar x) \le 0$, in contradiction with the fact that
\[
(-\Delta)^s u (\bar x) = \int_{\R^N} \frac{u(\bar x) -u(y)}{|x-y|^{N+2s}}\,dy > 0;
\]
here the strict inequality holds since $u \le a$ in $\R^N$, and $u<a$ in a set of positive measure by the boundary condition $u(x) \to 0$ as $|x| \to +\infty$. 

Now it remains to show that $\alpha_i<0$ for every $i$. This is a direct consequence of the Hopf's lemma for non-negative supersolutions proved in \cite{FallJarohs}, Proposition 3.3 plus Remark 3.5 therein. Indeed, we have already checked that for $c \in L^\infty(\R^N)$ defined in \eqref{def c in corol}, it results that 
\[
\begin{cases}
(-\Delta)^s (a-u) + c(x)(a-u) \ge 0 & \text{in $\R^N \setminus \overline{G}$} \\
a-u \ge 0 & \text{in $\R^N$}.
\end{cases}
\]
This implies $(\pa_\nu)_s (a-u) >0$ on $\pa G$, that is $(\pa_\nu)_s u<0$ on $\pa G$, where $\nu$ denotes the unit normal vector to $\pa G$ directed inwards $\R^N \setminus \overline{G}$.

\subsection{Proof of Corollary \ref{corol: subharmonicity} }

If there exists $x \in \R^N \setminus \overline{G}$ such that $u(x) > a$, then by the boundary conditions $u$ has an interior maximum point $ \bar x \in \R^N \setminus \overline{G}$. Therefore
\[
(-\Delta)^s u(\bar x) = \int_{\R^N} \frac{u(\bar x)-u(y)}{|\bar x-y|^{N+2s}}\,dy \ge 0.
\]
Since $f(u) \le 0$ in $\R^N$, this forces $(-\Delta)^s u (\bar x) = 0$, and in turn $u(x) = u(\bar x) = a$ for every $x \in \R^N$, in contradiction with the fact that $u(x) \to 0$ as $|x| \to +\infty$.

\section{Overdetermined problems in annular sets}\label{sec: over annular}

The strategy of the proof of Theorem \ref{thm: main 2} is similar to that of Theorem \ref{thm: main 1}. We apply the moving planes method to show that for any direction $e \in \mathbb{S}^{N-1}$ there exists $\bar \lambda=\bar \lambda(e)$ such that both the sets $G$ and $\Omega$, and the solution $u$, are symmetric with respect to the hyperplane $T_{e,\bar \lambda(e)}$. We fix at first $e=e_N$ and, for $\lambda \in \R$, 
we let $T_\lambda, H_\lambda, x^\lambda,\dots$ be defined as in \eqref{notation}. We only modify the definitions of $\Sigma_\lambda$ in the following way:
\[
\Sigma_{\lambda} := (\Omega \cap H_\lambda) \setminus \overline{G^\lambda}.
\]
Furthermore, instead of $d$ and $\bar \lambda$ we define 
\begin{align*}
d_G &:= \inf\{ \lambda \in \R: \text{$T_\mu \cap \overline{G}= \emptyset$ for every $\mu > \lambda$}\} \\
d_\Omega &:= \inf\{ \lambda \in \R: \text{$T_\mu \cap \overline{\Omega}= \emptyset$ for every $\mu > \lambda$}\} \\
\bar \lambda_G &:= \inf\left\{ \lambda \le d_G\left| \begin{array}{l}
\text{$(\overline{G} \cap H_\mu)^\mu \subset G \cap H_\mu^\mu$ with strict inclusion, and} \\
\text{$\langle \nu(x),e_N \rangle >0$ for every $x \in T_\mu \cap \pa G$, for every $\mu > \lambda$}
\end{array} \right.\right\} \\
\bar \lambda_\Omega &:= \inf\left\{ \lambda \le d_\Omega\left| \begin{array}{l}
\text{$(\overline{\Omega} \cap H_\mu)^\mu \subset \Omega \cap H_\mu^\mu$ with strict inclusion, and} \\
\text{$\langle \nu(x),e_N \rangle >0$ for every $x \in T_\mu \cap \pa \Omega$, for every $\mu > \lambda$}
\end{array} \right.\right\} \\
\bar \lambda&:= \max \{\bar \lambda_G, \bar \lambda_\Omega\}.
\end{align*}
Note that $\bar \lambda_G$ and $\bar \lambda_\Omega$ are the critical positions for $G$ and $\Omega$, respectively, and $\bar \lambda$ can be considered as a critical position for $\Omega \setminus \overline{G}$. As in the previous section, we start with a simple geometric observation.

\begin{lemma}\label{lem: geom remark bdd}
The following properties hold:
\begin{itemize}
\item[($i$)] for $\lambda \ge \bar \lambda_G$, the set $\overline{G} \cap H_\lambda$ is convex in the $e_N$ direction;
\item[($ii$)] for $\lambda \ge \bar \lambda_{\Omega}$, the set $\overline{\Omega} \cap H_\lambda$ is convex in the $e_N$ direction.
%\item[($iii$)] for $\lambda \ge \bar \lambda$, the set $(\overline{\Omega} \setminus \overline{G})\cap H_\lambda$ is convex in the $e_N$ direction.
\end{itemize}
\end{lemma}

The proof is analogue to that of Lemma \ref{lem: geom remark}, and thus is omitted.

For $w_\lambda(x):= u(x^\lambda)-u(x)$, we have that 
\[
(-\Delta)^s w_\lambda+c_\lambda(x) w_\lambda = 0 \qquad \text{in $\Sigma_\lambda$},
\]
exactly as in the previous section (we refer to \eqref{def: c_lambda} for the definition of $c_\lambda$). In the first part of the proof, we aim at showing that the set $\Lambda$ defined by 
\[
\Lambda:=\left\{ \lambda \in (\bar \lambda, d_\Omega): \text{$w_\mu \ge 0$ in $\Sigma_\mu$ for every $\mu \ge \lambda$} \right\}
\]
coincides with the interval $(\bar \lambda,d_{\Omega})$, that $w_{\lambda}>0$ for every $\lambda \in \Lambda$, and that $w_{\bar \lambda} \equiv 0$ in $\Sigma_{\bar \lambda}$. Since we are not assuming that $f$ is monotone (not even for small value of its argument), the argument in Lemma \ref{lem: moving initial} does not work. Nevertheless, we can take advantage of the boundedness of $\Omega$ to apply the maximum principle in domain of small measure.

\begin{lemma}\label{lem: moving initial bdd}
There exists $\sigma >0$ such that $w_{\lambda} \ge 0$ in $\Sigma_\lambda$ for every $\lambda \in (d_{\Omega}-\sigma,d_{\Omega})$.
\end{lemma}
\begin{proof}
Since $f$ is Lipschitz continuous and $u$ is bounded, there exists $c_\infty>0$ independent of $\lambda$ such that $\|c_{\lambda}\|_{L^\infty(\R^N)} \le c_\infty$. Then, it is well defined, and independent on $\lambda$, the value $\delta=\delta(N,s,c_\infty)$ as in Proposition \ref{SMALL}. For $\lambda$ a little smaller than $d_{\Omega'}$, the measure of $\Sigma_\lambda$ is smaller than $\delta$, and the function $w_\lambda$ satisfies
\[
\begin{cases}
(-\Delta)^s w_\lambda + c_\lambda(x) w_\lambda = 0 & \text{in $\Sigma_{\lambda}$} \\
w_\lambda(x) = - w_\lambda(x^\lambda) \\
w_{\lambda} \ge 0 & \text{in $H_\lambda \setminus \Sigma_\lambda$}.
\end{cases}
\] 
As a consequence, by Proposition \ref{SMALL} we deduce that $w_\lambda \ge 0$ in $\Sigma_\lambda$.
\end{proof}

This means that the hyperplane $T_\lambda$ moves and reaches a position $\mu = \inf \Lambda <d_\Omega$. We aim at showing that $\mu = \bar \lambda$. This is the object of the next two lemmas.

\begin{lemma}\label{lem: condizione simmetria bdd}
Let $\lambda \ge \bar \lambda$. If $w_{\lambda}(x) =0$ for $x \in \Sigma_{\lambda}$, then both $G$ and $\Omega$ are symmetric with respect to $T_\lambda$. In particular, if $\lambda>\bar \lambda$, then $w_\lambda \ge 0$ in $\Sigma_\lambda$ implies $w_\lambda>0$ therein.
\end{lemma}
\begin{proof}
By the strong maximum principle, if $w_{\lambda}(x) =0$ for $x \in \Sigma_{\lambda}$, then $w_\lambda \equiv 0$ in $H_\lambda$. As in the proof of Lemma \ref{lem: condizione per simmetria}, this implies that $G$ is symmetric with respect to $T_\lambda$, that is, $\lambda= \bar \lambda_G$. It remains to show that $\lambda$ is also equal to $\bar \lambda_\Omega$, and $\Omega$ is symmetric with respect to $T_\lambda$. To this aim, we observe that if this is not the case, then 
\[
F:= (\overline{\Omega} \cap H_\lambda^\lambda) \setminus (\overline{\Omega} \cap H_\lambda)^\lambda \neq \emptyset.
\]
Therefore, if $x \in F^\lambda \subset H_\lambda$, we have 
\begin{align*}
 x^\lambda \in F \subset \overline{\Omega} \quad \Longrightarrow \quad  u(x^\lambda) > 0 \\
 x \in \R^N \setminus \overline{\Omega} \quad \Longrightarrow \quad u(x) = 0,
 \end{align*}
and hence $w_\lambda(x) >0$, a contradiction. Then also $\Omega$ is symmetric with respect to $T_\lambda$, which forces $\lambda=\bar \lambda_\Omega$.
\end{proof}

\begin{lemma}\label{4.4}
There holds $\Lambda=(\bar \lambda,d_\Omega)$.
\end{lemma}
\begin{proof}
By contradiction, suppose that~$\mu=\inf \Lambda>\bar \lambda$. Differently from the previous section, we use the again maximum principle in sets of small measure. Let $\delta$ as in Proposition \ref{SMALL} (we have already observed in the proof of Lemma \ref{lem: moving initial bdd} that $\delta$ can be chosen independently of $\lambda$). By Lemma \ref{lem: condizione simmetria bdd}, $w_\mu >0$ in $\Sigma_\mu$. Thus, there exists a compact set $K \Subset \Sigma_{\mu}$ such that 
\[
|\Sigma_\mu \setminus K| < \delta/2 \quad \text{and} \quad w_{\mu} \ge C >0 \text{ in $K$},
\]
where we have used the boundedness of $\Sigma_\mu$. Furthermore, observing that by continuity $\Sigma_{\lambda} \to \Sigma_{\mu}$ as $\lambda \to \mu$, we can suppose that $K \Subset \Sigma_{\lambda}$ provided $|\lambda-\mu|<\eps$ for some $\eps>0$ sufficiently small. If necessary replacing $\eps$ with a smaller quantity, by continuity again it follows that
\[
|\Sigma_\lambda \setminus K| < \delta \quad \text{and} \quad w_{\lambda} \ge \frac{C}{2} \text{ in $K$}
\]
whenever $|\lambda-\mu|<\eps$. For such values of $\lambda$, in the remaining part $\tilde \Sigma_{\lambda}= \Sigma_\lambda \setminus K$ we have 
\[
\begin{cases}
(-\Delta)^s w_{\lambda}+ c_\lambda w_\lambda = 0 & \text{in $\tilde \Sigma_\lambda$} \\
w_\lambda(x) = -w_\lambda(x^\lambda) \\
w_\lambda \ge 0 & \text{in $H_\lambda \setminus \tilde \Sigma_\lambda$},
\end{cases}
\]
which means that we are in position to apply Proposition \ref{SMALL}, deducing that $w_\lambda \ge 0$ in $H_\lambda$. In particular, for $\mu-\eps<\lambda \le \mu$ we obtain $w_\lambda >0$ in $\Sigma_\lambda$ thanks to Lemma \ref{lem: condizione simmetria bdd}, in contradiction with the minimality of $\mu$.  
\end{proof}

\begin{proof}[Conclusion of the proof of Theorem \ref{thm: main 2}]
We proved in Lemma~\ref{4.4}
that $\Lambda=(\bar \lambda,d_{\Omega})$. Hence, by Lemma \ref{lem: condizione simmetria bdd}, to obtain the symmetry of $G$ and of $\Omega$, it is sufficient to check that $w_{\bar \lambda} \equiv 0$ in $H_{\bar \lambda}$. As in the proof of Theorem \ref{thm: main 1}, we argue by contradiction assuming that $w_{\bar \lambda} >0$ in $\Sigma_{\bar \lambda}$.
Note that the critical position $\bar \lambda$ can be reached for four possible reasons: internal tangency for $G$, internal tangency for $\Omega$, orthogonality condition for $G$, orthogonality condition for $\Omega$. In all such cases we can reach a contradiction exactly as in the conclusion of the proof of Theorem \ref{thm: main 1}. This proves that both $\Omega$ and $G$ are symmetric with respect to $T_{\bar \lambda}$, and by Lemma \ref{lem: geom remark bdd}, they are also convex in the $e_N$ direction. If $\Omega$ has several two connected components $\Omega_1$ and $\Omega_2$, then by convexity the only possibility is that $\Omega_1$ and $\Omega_2$ are aligned along $T_{\bar \lambda}$. But in this case we can obtain a contradiction as in the conclusion of the proof of Theorem \ref{thm: main 1}. On $G$ we can argue exactly in the same way.
\end{proof}

\subsection{Proof of Theorem \ref{thm: main 2 prime}}

The proof differs only for some details from that of Theorem \ref{thm: main 1 prime}, and thus it
is only sketched. First of all, by monotonicity $c_\lambda \ge 0$ in $\R^N$ for every $\lambda$, and hence by Proposition 3.1 in \cite{FallJarohs} (weak maximum principle for anti-symmetric functions) we have that $w_\lambda \ge 0$ in $\Sigma_\lambda$ for every $\lambda \ge \bar \lambda$. Moreover, as in the conclusion of the proof of Theorem \ref{thm: main 2}, $w_{\bar \lambda} \equiv 0$ in $H_{\bar \lambda}$. Repeating the same argument for any direction $e \in \mathbb{S}^{N-1}$, we deduce by Proposition \ref{prop: criterion} that $u$ is radially symmetric and radially non-increasing in $\R^N$, which implies that $\{u=a\}=B_1$ and $\{u>0\}=B_2$ are concentric balls, and $B_1 \subset B_2$. By monotonicity, $G \subset B_1$ and $B_2 \subset \Omega$. Arguing as in the proof of Theorem \ref{thm: main 1 prime}, we deduce that $G=B_1$ and $\Omega=B_2$.

\section{Radial symmetry}\label{sec: radial}

\subsection{Proof of Theorem \ref{thm: radial 1}}
We briefly describe how the proof of Theorem \ref{thm: main 1} can be adapted to obtain Theorem \ref{thm: radial 1}. Without loss of generality, we suppose that $x_0$,  the centre of the cavity, is $0$. Using the same notation introduced in Section \ref{sec: over exterior}, see \eqref{notation}, we observe that for any direction $e \in \mathbb{S}^{N-1}$ the critical position $\bar \lambda(e)$ is reached for $\lambda=0$.  Let us fix $e=e_N$, and let us introduce
\[
\Lambda:= \left\{ \lambda \ge 0: \text{$w_\mu \ge 0$ in $\Sigma_\mu$ for every $\mu \ge \lambda$} \right\}.
\]
We aim at proving that $\Lambda=(0,+\infty)$, and that $w_\lambda > 0$ in $\Sigma_\lambda$ for every $\lambda>0$. Once that this is proved, we can repeat the argument with $e=-e_N$. Since the critical position for $e_N$ and $-e_N$ is the same, we have 
\[
T_{e_N,\bar \lambda(e_N)} = T_{-e_N, \bar \lambda(-e_N)} = \{x_N=0\},
\]
from which we infer that $u$ is symmetric with respect to $\{x_N=0\}$, and strictly decreasing in the $x_N$ variable outside $B_{\rho}(0)$. Symmetry and monotonicity in all the other directions can be obtained in the same way.

As in Lemma \ref{lem: moving initial}, we can show that $\Lambda \neq \emptyset$. Once that this is done, as in Lemmas \ref{lem: condizione per simmetria} and \ref{lem: moving continuation}, we can show that $\Lambda=(0,+\infty)$. This completes the proof. Notice that assumption \ref{cond Neumann} is used in Lemma \ref{lem: moving continuation}, case 3).

\begin{remark}
When $\Omega$ is bounded, by using the method in Section \ref{sec: over annular}, we see that the conclusion of Theorem \ref{thm: radial 1} remains true if we assume that only one between $\Omega$ and $G$ is a ball, and prescribe constant $s$-Neumann boundary condition on the other component. For the same result in the local case, we refer to \cite[Theorem 5]{Sirakov}. 
\end{remark} 

\subsection{Proof of Theorem \ref{thm: radial point}}

Without loss of generality, we suppose that $x_0=0$. Using the same notation introduce in Section \ref{sec: over exterior}, we fix $e=e_N$ and observe that 
\[
\Sigma_\lambda=H_\lambda \setminus \{(0',2\lambda)\} \qquad \forall \lambda >0,
\]
and $\bar \lambda=0$ is the critical position for the hyperplane $T_\lambda$. We slightly modify the definition of $\Lambda$ in the following way:
\[
\Lambda:= \left\{ \lambda>0: \text{$w_\mu >0$ in $\Sigma_\mu$ for every $\mu>\lambda$} \right\},
\]
where we recall that $w_\lambda(x) = u(x^\lambda)-u(x)$ satisfies the equation
\[
(-\Delta)^s w_\lambda+c_\lambda(x) w_\lambda = 0 \qquad \text{in $\Sigma_\lambda$},
\]
and $c_\lambda$ has been defined in \eqref{def: c_lambda}. We aim at showing that $\Lambda=(0,+\infty)$. This is the object of the next three lemmas.

%This yields monotonicity in $x_N$ in the half-space $\{x_N>0\}$, and by continuity also $w_0\ge 0$ in $\{x_N >0\}$. Arguing in the same way with $e_N$ replaced by $-e_N$, we obtain also the opposite inequality, so that $u$ turns out to be symmetric with respect to $T_0$, and monotone decreasing in $x_N$ in the half-space $\{x_N >0\}$.

\begin{lemma}
There exists $R>0$ sufficiently large such that $w_\lambda>0$ in $\Sigma_\lambda$ for every $\lambda>R$.
\end{lemma}
\begin{proof}
Exactly as in Lemma \ref{lem: moving initial}, it is possible to show that $w_\lambda \ge 0$ in $\Sigma_\lambda$. By the strong maximum principle, Proposition \ref{STRONG}, either $w_\lambda>0$ in $\Sigma_\lambda$ or $w_\lambda \equiv 0$ in $H_\lambda$. 
%Notice that Proposition \ref{STRONG} is applicable, because as observed $c_\lambda$ is bounded. It remains to prove that, for $\lambda$ large enough, only the former alternative can occur. 
If $w_\lambda \equiv 0$ in $H_\lambda$, then $u(0^\lambda) = u(0',2\lambda) = a$. On the other hand, since $u(x) \to 0$ as $|x| \to +\infty$ we have that for $\lambda$ very large $u(x) \le a/2$ in the whole half-space $H_\lambda$, a contradiction. 
\end{proof}

Thus, the quantity $\mu:= \inf \Lambda \ge 0$ is a real number.

\begin{lemma}\label{lem: monot}
The function $u$ is monotone strictly decreasing in $x_N$ in the half-space $\{x_N > \mu\}$.
\end{lemma}
\begin{proof}
Let $y,z \in \{x_N > \mu\}$ with $y_N<z_N$. We aim at showing that $u(y) > u(z)$. For $\lambda:= (y_N+z_N)/2$, we claim that $z \in \Sigma_\lambda$. Once that this is shown, the desired conclusion simply follows by the fact that 
\[
u(y)-u(z) = w_\lambda(z) >0,
\]
as $\lambda>\mu$. Since $z_N>\lambda$, if $z \not \in \Sigma_\lambda$, then necessarily $z=0^\lambda$. This means that $y=0$, in contradiction with the fact that $y_N > \mu  \ge 0$. 
\end{proof}

We are ready to complete the proof of Theorem \ref{thm: radial point} by showing that $\mu=0$.

\begin{lemma}
It holds $\Lambda=(0,+\infty)$.
\end{lemma}
\begin{proof}
By contradiction, let $\mu>0$. At first, by continuity $w_\mu \ge 0$ in $\Sigma_\mu$. 
Thus, by the strong maximum principle we have that either $w_\mu >0$ in $\Sigma_\mu$, or $w_\mu \equiv 0$ in $H_\mu$. To rule out the latter alternative, we observe that having assumed $\mu>0$, we obtain $u(0',2\mu) = a$. Thanks to the previous lemma, we infer that $u(0',x_N) > a$ whenever $x_N \in (\mu,2\mu)$, in contradiction with the maximality of $a$. 

Thus, it remains to reach a contradiction when $w_\mu > 0$ in $\Sigma_\mu$. By the definition of $\inf$, there exist sequences $0<\lambda_k<\mu$ and $x_k \in \Sigma_{\lambda_k}$ such that $\lambda_k \to \mu$ and $w_{\lambda_k}(x_k)<0$. Since $w_{\lambda_k} \ge 0$ in $\pa \Sigma_{\lambda_k}$ and tends to $0$ as $|x| \to +\infty$, it is not restrictive to assume that $x_k$ is an interior minimum points for $w_{\lambda_k}$ in $\Sigma_{\lambda_k}$. If $|x_k| \to +\infty$, we obtain a contradiction as in Lemma \ref{lem: moving initial}. Hence, up to a subsequence $x_k \to \bar x\in \overline{\Sigma_\mu}$. Notice that by uniform convergence $w_\mu(\bar x) = 0$, which forces $\bar x \in \partial \Sigma_\mu$. If $\bar x= (0',2\mu)$, this means that $u(0',2\mu) = a$, and by Lemma \ref{lem: monot} we obtain a contradiction with the fact that $u \le a$ in $\R^N$. Therefore $\bar x \in T_\mu$. This means that all the points $x_k$, and also $\bar x$, are interior points for the anti-symmetric functions $w_{\lambda_k}$ and $w_\mu$ in the sets
\begin{align*}
\R^N \setminus \left( \{0\} \cup \{(0',2\lambda_k)\}\right) = \Sigma_{\lambda_k} \cup T_{\lambda_k} \cup \Sigma_{\lambda_k}^{\lambda^k} \\
\R^N \setminus \left( \{0\} \cup \{(0',2\mu)\} \right) = \Sigma_{\mu} \cup T_{\mu} \cup \Sigma_{\mu}^{\mu},
\end{align*}
respectively. As a consequence, we can argue as in case 1) of the proof of Lemma \ref{lem: moving continuation}, deducing that for some $\rho,\gamma>0$ the sequence $\{w_k\}$ is uniformly bounded in $\mathcal{C}^{1,\gamma}(\overline{B_\rho(\bar x)})$. This entails $\mathcal{C}^1$ convergence in $B_{\rho}(\bar x)$, and by minimality $\nabla w_\mu (\bar x) = 0$, in contradiction with Proposition \ref{prop: new hopf}. 
\end{proof}

\section{Existence results}\label{sec: existence}

This section is devoted to the proof of the existence of a solution to 
\begin{equation}\label{ex problem}
\begin{cases}
(-\Delta)^s u = f(u) & \text{in $\R^N \setminus \overline{B_1}$} \\
u = a & \text{in $\overline{B_1}$},
\end{cases}
\end{equation}
satisfying all the assumptions of Theorem \ref{thm: main 1 prime}. To this aim, we recall that the critical exponent for the embedding $H^s(\R^N) \hookrightarrow L^p(\R^N)$ is defined as $2^*_s:= 2N/N-2s$. Let~$
{\mathcal{R}}^s$
be the space of functions
in~$H^s(\R^N)$
that are $u$ radial and radially decreasing with respect to the origin.
We point out that, if~$u \in {\mathcal{R}}^s$, the decay estimate
\begin{equation}\label{DDCC}
|u(x)| \le C \|u\|_{L^2(\R^N)} |x|^{-N/2} \qquad \forall x \in \R^N
\end{equation}
holds (see e.g. Lemma 2.4 in \cite{DipPalVal} for a simple proof), and this
ensures that $u(x) \to 0$ as $|x| \to +\infty$.
Moreover, it is known\footnote{The details of the easy proof
of this compactness statement can be obtained as
follows. Given a bounded family $F$ in ${\mathcal{R}}^s$
and
fixed $\epsilon>0$, we find an $\epsilon$-net for~$F$.
That is, first we use~\eqref{DDCC}
to say that for any $u\in F$,
we have that $\|u\|_{L^p(\R^N\setminus B_R)}<\epsilon/2$
if $R$ is chosen suitably large (in dependence of $\epsilon$).
Then we use local compact embeddings (see e.g.
Corollary 7.2 in \cite{MR2944369})
for the compactness in $L^p(B_R)$: accordingly, we find
$h_1, \dots h_M\in L^p(B_R)$ such that for any $u\in F$ there
exists $i\in\{ 1,\dots,M\}$ such that $\| u-h_i\|_{L^p(B_R)}
<\epsilon/2$ (of course $M$ may also depend on $\epsilon$).
We extend $h_i$ as zero outside $B_R$, and we found
an $\epsilon$-net since in this way
$\| u-h_i\|_{L^p(\R^N)}\leq \| u-h_i\|_{L^p(B_R)}+
\| u\|_{L^p(\R^N\setminus B_R)}<\epsilon$.
This shows the compactness that we need.
For a more general and comprehensive treatment of this topic,
we refer to \cite{MR683027} and Theorem 7.1 in~\cite{PRPR}.
}
that~${\mathcal{R}}^s$
compactly embeds into $L^p(\R^N)$ for every $2 < p < 
2^*_s$.

\begin{theorem}
Let $f(t):= g(t) -t$ for some $g: \R \to \R$ continuous, odd, and such that
\begin{equation}\label{ass existence}
g(t) t \le 0  \qquad \text{for every $t \in \R$}.
\end{equation}
Then there exists a radially symmetric and radially decreasing solution $u \in \mathcal{C}^s(\R^N)$ of problem \eqref{ex problem}, satisfying the additional condition $0 \le u \le a$ in $\R^N$.
\end{theorem}

\begin{proof}
We denote $\Omega = \R^N \setminus \overline{B_1}$, and set
\[
X:= \left\{ u \in H^s(\Omega): \text{$u = a$ a.e. in $B_1$} \right\}.
\]
Let $J:X \to \R$ be defined by
\[
J(u) := \frac{1}{2} \int_{\R^{2N}} \frac{(u(x)-u(y))^2}{|x-y|^{N-2s}}\,dx dy + \frac{1}{2}\int_{\Omega} u^2 - \int_{\Omega} G(u),
\]
where $G$ denotes a primitive of $g$. It is not difficult to check that if $u \in X$ is a minimizer for $J$, then $u$ solves \eqref{ex problem}, and hence in the following we aim at proving the existence of such a minimizer. Let $c:= \inf_{X} J$. Since $G \le 0$ by assumption \eqref{ass existence}, we have that
\[
J(u) \ge \frac{1}{2} \int_{\R^{2N}} \frac{(u(x)-u(y))^2}{|x-y|^{N-2s}}\,dx dy + \frac{1}{2}\int_{\Omega} u^2 = \frac{1}{2} \|u\|_{H^s(\Omega)}^2,
\]
which implies that $J$ is coercive on $H^s(\Omega)$. Let $\{u_n\}$ be a minimizing sequence for $c$. Since $g$ is odd we can suppose that $u_n \ge 0$ for every $n$ (recall that, if $u \in H^s(\Omega)$, also $|u| \in H^s(\Omega)$), and thanks to the fractional Polya-Szego inequality (see \cite{Park}) it is not restrictive to assume that each $u_n$ is radially symmetric and radially non-increasing with respect to $0$. Thus $\{u_n\}$ is a bounded sequence in ${\mathcal{R}}^s$, so that by compactness we can extract a subsequence of $\{u_n\}$ (still denoted $\{u_n\}$), and find a function $u \in {\mathcal{R}}^s$, such that $u_n \to u$ weakly in $H^s(\Omega)$ and a.e. in $\R^N$. Notice that $u \in X$. Now by weak lower semi-continuity we infer
\[
c \le J(u) \le \liminf_{n \to +\infty} J(u_n) = c,
\]
namely $u$ is a minimizer for $c$, and hence a solution of \eqref{ex problem}. By convergence, it is radially symmetric and radially non-increasing with respect to $0$, and is nonnegative. Moreover, 
\[
\begin{cases}
 (-\Delta)^s u = g(u) -u \le 0 & \text{in $\R^N \setminus \overline{B_1}$}\\
u \le a & \text{in $\overline{B_1}$} \\
u(x) \to 0& \text{as $|x| \to +\infty$},
\end{cases}
\]  
which, as in the proof of Corollary \ref{corol: subharmonicity}, implies $u \le a$ in $\R^N$. Finally, by Theorem \ref{thm: regularity} it results $u \in \mathcal{C}^s(\R^N)$.
\end{proof}

%A second relevant case where the existence of solutions can be easily proved is the $s$-harmonic one.
%
%\begin{theorem}
%The problem
%\[
%\begin{cases}
%(-\Delta)^s u = 0 &\text{in $\R^N \setminus \overline{B_1}$} \\
%u = a & \text{in $\overline{B_1}$}
%\end{cases}
%\]
%has a radially symmetric and radially decreasing solution $u$, satisfying the additional condition $0 \le u \le a$ in $\R^N$ and $u(x) \to 0$ as $|x| \to +\infty$.
%\end{theorem}
%\begin{proof}
%For any $R>1$, let $u_R$ be the $s$-harmonic function in $B_R \setminus \overline{B_1}$ satisfying the boundary conditions $u=a$ in $\overline{B_1}$ and $u=0$ in $\R^N \setminus B_R$. By the maximum principle, we have $0<u<a$ in the annulus, so that $\{u_R\}$ is bounded in $L^\infty(\R^N)$, and thanks to Proposition 5 in \cite{SerValViscosity} it is locally uniformly convergent, up to a subsequence, to a limit function $u$ such that $0 \le u \le a$ in $\R^N$. By Corollary \ref{corol: radial}, it is also radially symmetric and radially decreasing. It remains to check that $u(x) \to 0$ as $|x| \to +\infty$. To this aim, we set $v(x) := a |x|^{2s-N}$, and we observe that
%\[
%\begin{cases}
%(-\Delta)^s(v-u_R) = 0 & \text{in $B_R \setminus \overline{B_1}$} \\
%v - u_R \ge 0 & \text{in $\R^N \setminus (B_R \setminus \overline{B_1})$}.
%\end{cases}
%\]
%From this, as in the conclusion of the previous proof we infer $v \ge u_R$ in $\R^N$ for every $R$, and by uniform convergence the thesis follows.
%\end{proof}

\appendix

\section{Regularity of weak solutions in unbounded domains}

In this section we discuss the regularity of weak solutions of
\begin{equation}\label{pb: regularity}
\begin{cases}
(-\Delta)^s w= g & \text{in $\Omega$} \\
w = 0 & \text{in $\R^N \setminus \Omega$},
\end{cases}
\end{equation}
where $\Omega$ is an arbitrary open set of $\R^N$, neither necessarily bounded, nor necessarily connected. When $\Omega$ is bounded domain, the regularity of weak solutions has been studied in \cite{SeR-O, SerValViscosity}, and in what follows we show how to adapt the arguments therein to deal with the general case considered here. Notice that if $u$ is a bounded solution of the first two equations in \eqref{pb: overdet}, then the difference $w=a-u$ solves a problem of type \eqref{pb: regularity}.

\begin{theorem}\label{thm: regularity}
Let $\Omega$ be an arbitrary open set of class $\mathcal{C}^{1,1}$, and let $g \in \mathcal{C}(\R^N) \cap L^\infty(\R^N)$. If $w \in L^\infty(\R^N)$ is a bounded weak solution of \eqref{pb: regularity}, then $w \in \mathcal{C}^s(\R^N) \cap \mathcal{C}^{1,\sigma}(\Omega)$ for some $\sigma \in (0,1)$, and the integral representation \eqref{integral representation} holds point-wise. Moreover, if $\delta$ denotes the distance from the boundary of $\Omega$, then $w/\delta^s \in \mathcal{C}^{0,\gamma}(\overline{\Omega})$ for some $\gamma \in (0,1)$.
\end{theorem}

We recall that the definition of weak solution has been given in Section \ref{sec: preliminaries}.
\begin{proof}
We wish to prove that $w$ is a viscosity solution of \eqref{pb: regularity}, so that the regularity theory for viscosity solutions, developed in \cite{CafSil, CafSil2, BRR}, gives the desired result. 

We show that $w \in \mathcal{C}(\R^N)$. Once that this is done, the proof can be concluded following the argument in \cite[Theorem 1]{SerValViscosity} or \cite[Remark 2.11]{SeR-O}. 

By Proposition 5 in \cite{SerValViscosity}, $w \in \mathcal{C}(\Omega)$, and therefore it remains to rule out the possibility that $w$ has a discontinuity on $\pa \Omega$. To this aim, we argue by contradiction assuming that there exists a point of discontinuity $x_0 \in \partial \Omega$ for $w$. Let $\varphi \in H^s_{\loc}(\R^N)$ be such that 
\[
\begin{cases}
(-\Delta)^s \varphi \ge 1 & \text{in $B_4 \setminus B_1$} \\
\varphi \equiv 0 & \text{in $\overline{B_1}$} \\
0 \le \varphi \le C(|x|-1)^s & \text{in $B_4 \setminus B_1$} \\
1 \le \varphi \le C & \text{in $\R^N \setminus B_4$}
\end{cases}
\]
for some $C>0$. The existence of $\varphi$ has been proved in \cite[Lemma 2.6]{SeR-O}. Since $\Omega$ satisfies an exterior sphere condition, there exists a ball $B_{\rho_0} \subset \R^N \setminus \Omega$ which is tangent to $\Omega$ in $x_0$. By scaling and translating $\varphi$, we find an upper barrier for $w$ in $B_{4 \rho_0} \setminus \overline{B_{\rho_0}}$, vanishing in $\overline{B_{\rho_0}}$ (here we use the fact that both $g$ and $w$ are in $L^\infty(\R^N)$). Arguing in the same way on $-w$, we deduce that $|w|<C \delta^s$ in $\Omega \cap B_{4 \rho_0}$, yielding in particular $w(x) \to 0$ as $x \to x_0$. This implies that $w$ is continuous in $x_0$, a contradiction. 

Therefore $w \in \mathcal{C}(\R^N)$ and, as observed, the desired regularity results follow. As far as the validity of the integral representation is concerned, we first notice that any $L^\infty(\R^N)$ weak solution is also a distributional solution, according to the definition in Section 2 in \cite{Silv}. We can then apply Proposition 2.8 therein to deduce that, in case $s>1/2$, the exponent $\sigma$ is strictly larger than $2s-1$ (we observe that, while in \cite{Silv} the result is stated in the whole space $\R^N$, the proof is local, and then can be adapted to our setting with minor changes). As a consequence, Proposition 2.4  in \cite{Silv} yields the validity of the integral representation \eqref{integral representation}.
\end{proof}

\end{document}